\documentclass[USenglish]{article}
\usepackage[utf8]{inputenc}
\usepackage[english]{babel}
\usepackage{graphicx}
\usepackage{amsmath}
\usepackage{cleveref}
\usepackage{xcolor}

\usepackage[a4paper, margin=2.5cm]{geometry}

\usepackage{microtype}

\usepackage[%
format=plain,           
labelsep=colon,               
justification=justified,      
singlelinecheck=true,         
labelfont=bf                  
]{caption}

\usepackage{subcaption}

\usepackage{amsmath}
\usepackage{amsthm}
\usepackage{amsfonts}
\usepackage{amssymb}
\usepackage{amstext}

\usepackage{mathtools}

\usepackage{physics}

\usepackage{multirow}

\renewcommand{\epsilon}{\varepsilon}

\usepackage{siunitx}
\newtheorem{theorem}{Theorem}[section]
\newtheorem{corollary}{Corollary}[theorem]
\newtheorem{lemma}[theorem]{Lemma}

\usepackage{abstract}

\usepackage{authblk}

\title{\huge Multi-level Parareal algorithm with Averaging for Oscillatory Problems}
\author[1]{Juliane Rosemeier}
\author[2]{Terry Haut}
\author[3]{Beth Wingate}
\affil[1]{University of Exeter}
\affil[2]{Lawrence Livermore National Laboratory}
\affil[3]{University of Exeter}
\date{\today}

\begin{document}

\maketitle

\renewcommand\abstractname{Abstract} 
\begin{abstract}
    The present study is an extension of the work done in  \emph{Parareal convergence for oscillatory pdes with finite time-scale separation} (2019), A. G. Peddle, T. Haut, and B. Wingate, \cite{Peddle_Haut_Wingate_2019}, and  \emph{An asymptotic parallel-in-time method for highly oscillatory pdes} (2014),  T. Haut and B. Wingate, \cite{Haut_Wingate_2014}, where a two-level Parareal method with averaging is examined. The method proposed in this paper is a multi-level Parareal method with arbitrarily many levels, which is not restricted to the two-level case. We give an asymptotic error estimate which reduces to the two-level estimate for the case when only two levels are considered. Introducing more than two levels has important consequences for the averaging procedure, as we choose separate averaging windows for each of the different levels, which is an additional new feature of the present study. The different averaging windows make the proposed method especially appropriate for multi-scale problems, because we can introduce a level for each intrinsic scale of the problem and adapt the averaging procedure such that we reproduce the behavior of the model on the particular scale resolved by the level. The computational complexity of the new method is investigated and the efficiency is studied on several examples.
\end{abstract}


\section{Introduction}
\label{sec:introduction}


 With complex changes in modern computer architectures comes new challenges for simulation and modeling to develop algorithms that can take advantage of their increased concurrency \cite{Lawrence_etAl2018}. The classic problem considered in this paper is the solution to partial differential equations that depend on space and time. One way to increase the computational performance of these problems is to increase the number of grid points in space, but often the time-step must be reduced to satisfy stability and/or accuracy constraints, like the CFL condition.  In this case, the serial time-stepping leads to longer model runs. This is one motivation for introducing parallelization to the time domain and is the central topic of this paper.

 In the present study, the focus is on the Parareal method, a time-parallel method first proposed in \cite{LionsEtAl2001}. Since its publication a lot of research has been done on the method with the aim to exploit the advantages of parallelism in time. For instance \cite{Peddle_Haut_Wingate_2019} and \cite{Haut_Wingate_2014} combine the Parareal method with averaging to solve fluid-dominated problems. Several studies treat the well-known stability issues related to problems of this type, see for instance \cite{Ruprecht_2018} or \cite{SteinerEtAl2015}. Convergence of the Parareal method is also discussed in \cite{Gander_Vandewalle_2007},
 \cite{Gander_Hairer_2008} or \cite{Bal_2005}.
Moreover, the authors of
\cite{Gander_et_al_2018} give several interpretations of the Parareal algorithm and especially show its relation to the MGRIT algorithm. 
Investigations on the error and convergence can be found for example in \cite{Southworth_2019}, \cite{Southworth_etAl_2021} or \cite{Friehoff_Southworth_2021}.There are also attempts to improve the understanding of hyperbolic or advection-dominated problems, see to \cite{De_Sterck_etAl_2020}, \cite{Hessenthaler_etAl_2020} or \cite{DeSterck_et_al_2021} just to name a few.

The problems under consideration in the present study exhibit scale separation and admit the following form
\begin{equation}
\label{eq:problem_under_consideration}
\frac{d \bf u}{dt} + \frac{1}{\varepsilon} \mathcal{L} {\bf u} =
\mathcal{N}({\bf u}) .
\end{equation}
The linear operator $\mathcal{L}$ is skew Hermitian, i.e. it has purely imaginary eigenvalues and is responsible for temporal oscillations in the solution. The parameter $\varepsilon$ is small and makes the system stiff. Especially, the system shows oscillatory stiffness. The term $\mathcal{N}$ is a quadratic non-linearity. In addition,  a diffusive term $\mathcal{D}$ can be added in equation \eqref{eq:problem_under_consideration}.

Applying the transformation
\begin{equation}
\label{eq:transormation}
   {\bf w}(t)=\exp(\frac{\mathcal{L}}{\varepsilon} t) {\bf u}(t)
\end{equation}to the above system, we can eliminate the linear term. The transformed system denoted as the  {\it modulation equation}, since it’s time evolution is more regular than equation \eqref{eq:problem_under_consideration}, admits the form
\begin{equation}
\label{eq:modulation_equation}
\frac{d\bf w}{dt}  = \exp\left(\frac{\mathcal{L}}{\varepsilon} t\right) \mathcal{N}\left(\exp\left(-\frac{\mathcal{L}}{\varepsilon} t\right) {\bf w} \right) .
 \end{equation}

There exist numerous scientific applications which have the form of equation \eqref{eq:problem_under_consideration}, including examples that occur in atmospheric and oceanic simulations, like the swinging spring \cite{Holm_Lynch_2002}, also called the elastic pendulum, or the rotating shallow water equations, see \cite{EmMa1996}.

As the operator $\mathcal{L}$ is skew Hermitian, the norm of the right-hand side of equation \eqref{eq:modulation_equation} is independent of $\epsilon$. Especially, applying the transformation \eqref{eq:transormation} to the problem \eqref{eq:problem_under_consideration} eliminated linear term and made the problem smoother. However taking higher order derivatives of $\bf w$ we see that with each order we get an additional power of $1/\epsilon$ in the derivative. 
To further mitigate the oscillatory stiffness we apply {averaging} techniques. Further mitigating the stiffness is important when numerical time-stepping schemes are applied, since the truncation error depends on higher order derivatives.


The basic idea of the averaging techniques applied in the present work is to replace an original problem which exhibits oscillatory stiffness by a problem which is less stiff. We might also say that the averaging smooths the modulation equation. 
In particular, the right-hand side of a time evolution problem, like equation \eqref{eq:modulation_equation}, is convolved using a scaled filter function $\rho$. For the scaling, an averaging window $\eta$ must be chosen appropriately. Particularly, it must mitigate the fast oscillations while leaving the coarse or mean behavior of the problem unaffected. In addition, the filter function $\rho$ must satisfy certain properties.  The ideas of temporal averaging were also investigated in related contexts such as ODEs \cite {Sanders_etal_2007}, in the context of heterogeneous multiscale methods (\cite{E_Enquist_03}, \cite{Engquist_Tsai_2005}), and in PDEs analysis \cite{EmMa1996}. In the present work averaging techniques shall be applied to modulation equations to construct {coarse propagators} for a Multi-level Parareal method.

In order to explain how the averaging mitigates the oscillatory stiffness due to fast oscillations, let us assume that a slow function is superimposed by a fast periodic function with zero mean. Integrating the fast function over an interval of length $\eta$ where $\eta$ is as large as a few times the period of the fast periodic function, the positive and negative contributions cancel each other. However, the integrand is weighted by a scaled kernel function $\rho$ with compact support, which decays fast close to the boundary of the compact support. Moreover, the knowledge of the exact period is not assumed in the method. Therefore, in general we do not observe an exact cancellation of the oscillations but rather a mitigation, see for example Figure 5. The technical details can be found in \cite{Engquist_Tsai_2005} in Lemma 2.2. Moreover, a mathematical formulation of the averaged equations is given in \cref{subsec:Two-level_Parareal}.

The Parareal method was first presented in \cite{LionsEtAl2001}. It is a time-parallel method with two levels. For the exposition in this study we enumerate the levels, i.e. for the two-level method we have level 1 and level 0. On level 1 a coarse time grid is introduced and on level 0 we have several fine time grids. On the coarse time grid of level 1 a coarse propagator, for instance a Runge-Kutta method, is applied to compute a numerical approximation to a differential equation. This numerical approximation is then improved iteratively using solutions computed with the fine propagator on level 0 in parallel. We can define the Multi-level Parareal algorithm recursively in the levels until we reach the two-level case.The idea of the Multi-level Parareal algorithm with $L$ levels is that the coarse propagator provides a solution on the coarsest level and the fine propagator is a  Multi-level Parareal algorithm with $L-1$ levels. 
More detailed descriptions of the two-level and multi-level methods can be found in \cref{subsec:Two-level_Parareal} and \cref{subsec:Two-level_Parareal_averaging}.

The strategy for the application of the Multi-level Parareal algorithm with averaging to a multi-scale problem can be formulated as follows: Suppose we are given a problem with several time scales, for example a modulation equation. For each scale we introduce a level, on which an averaged equation is solved. (Only on the finest level, level 0, we solve the full, unaveraged system.) We average such that we keep the features of the original system on that scale, but the finer components of the equation vanish through the averaging process. This step requires a convenient choice of the averaging window $\eta_l$, which is level-dependent in the multi-level case. The averaging procedure makes the system more well-behaved for the numerical time-stepping method. Thus, it gives us a good coarse propagator for that level. Combining this with the parallelization of the Parareal method shall result in efficient numerical algorithms.

The averaging process requires the formulation of  analytical equations. In particular, for each level we formulate an averaged, analytical equation which shows the same behavior as the original system, i.e. the modulation equation, on the coarser scales up to the scale that corresponds to the level considered, but whose behavior on the finer scales is different. Especially, the features on the finer scales are averaged in the averaged equation. The step of formulating analytical equations is skipped when other strategies to deal with the fast components, for example applying implicit methods, are used. However, the analytical equations provide descriptions of physical phenomena and are therefore a link to the theory or modeling of the physical application considered.

In this place, the importance of the modulation equation for the method shall be explained.
First, using the modulation equation is the first step which makes the equations smoother, but even after applying the averaging we can recover the oscillations in the solution, at least as they correspond to the linear term, by applying the inverse transformation. Second, the information about the phase is not in the initial condition anymore, but in the exponential explicitly in $t$ which is beneficial for numerical computations and shall be explained here. The coarse propagator of the Parareal method provides a solution with damped oscillations due to the application of an averaging procedure, see for example \cref{fig:secondComp_InitGuess}. However, we want to compute a solution to the unaveraged modulation equation with the full information about the oscillations. This is accomplished as the fine propagator computes a solution to the unaveraged modulation equation, see for instance \cref{fig:secondComp_ThirdIt}. When the fine propagator is applied the initial guesses, which are the initial values for the fine propagator, come from the coarse propagator and do not contain any of the variations from the fast oscillations. Still the fine propagator  computes the phase correctly. This indicates that the phase information is in the right-hand side in \eqref{eq:modulation_equation}, especially in the exponential explicitely in $t$ and not in the initial conditions. Further explanation can be found the numerical examples in \cref{subsec:System_with_three_scales}.

A major achievement of the present work is an asymptotic convergence proof. The error estimate for a two-level Parareal method is extended to the case of multiple levels with and without averaging. This is a generalization of the classical proof found in \cite{Gander_Hairer_2008} and the convergence proof for the APinT method found in \cite{Peddle_Haut_Wingate_2019}. The new proof has two main steps: In the mentioned literature error estimates for the two-level case can be found, however the fine propagator is assumed to be the exact propagator. When we introduce multiple levels we cannot make this assumption, because we want to know how the error contributions that emerge on the finer levels propagate through the different coarser levels and we will find that they are amplified by an amplification factor that depends on the details of the scheme. Consequently, the first step of the new proof is to refine the two-level estimate to the case where the fine propagator is not exact. The second main step is to apply an inductive argument to obtain an estimate for the multi-level case. This might be beneficial for identifying on which level the dominant error contribution emerges and therefore how the time-steps or the number of iterations on the different levels should be chosen to reach a certain error tolerance. For this idea  error estimators would be needed. This is a possible future extension of the presented work.

We want to answer the question if multi-level methods can be more efficient than two-level methods.  Let us suppose that we are given a two-level method. The aim is to design a multi-level method which is more efficient. When we do not do too many iterations with the multi-level method, the number of serial steps done with the multi-level method is less than number of serial steps done with the two-level method. However, we expect that we have to do several iterations with the multi-level method to reach the accuracy of the two-level method. This is illustrated in \cref{fig:complexity_error_multilevel_parareal}.  An example will be discussed in   \cref{subsec:RSWE}.

The paper is organized as follows: The algorithms are described in  \cref{sec:Formulation_of_the_algorithm}. Especially, the exposition in  \cref{subsec:Two-level_Parareal} is about the two-level Parareal schemes and  \cref{subsec:Two-level_Parareal_averaging} contains a description of the new multi-level schemes. Asymptotic convergence results with and without averaging can be found in  \cref{sec:Asymptotic_convergence}. In section \cref{sec:Efficency} the issue of computational complexity of the Multi-level Parareal schemes is discussed. Numerical examples can be found in the following  \cref{sec:Numerical_examples}. Finally, in  \cref{sec:Discussion_conclusion} a discussion of the results and a conclusion are given.


\section{Formulation of the algorithm}
\label{sec:Formulation_of_the_algorithm}

 The Parareal method was first formulated in \cite{LionsEtAl2001}. Versions with averaging incorporated can be found in \cite{Haut_Wingate_2014} or \cite{Peddle_Haut_Wingate_2019}. The next \cref{subsec:Two-level_Parareal} summarizes  the methods. Then, the  \cref{subsec:Two-level_Parareal_averaging} presents multi-level versions. 

\begin{figure}
    \centering
    \includegraphics{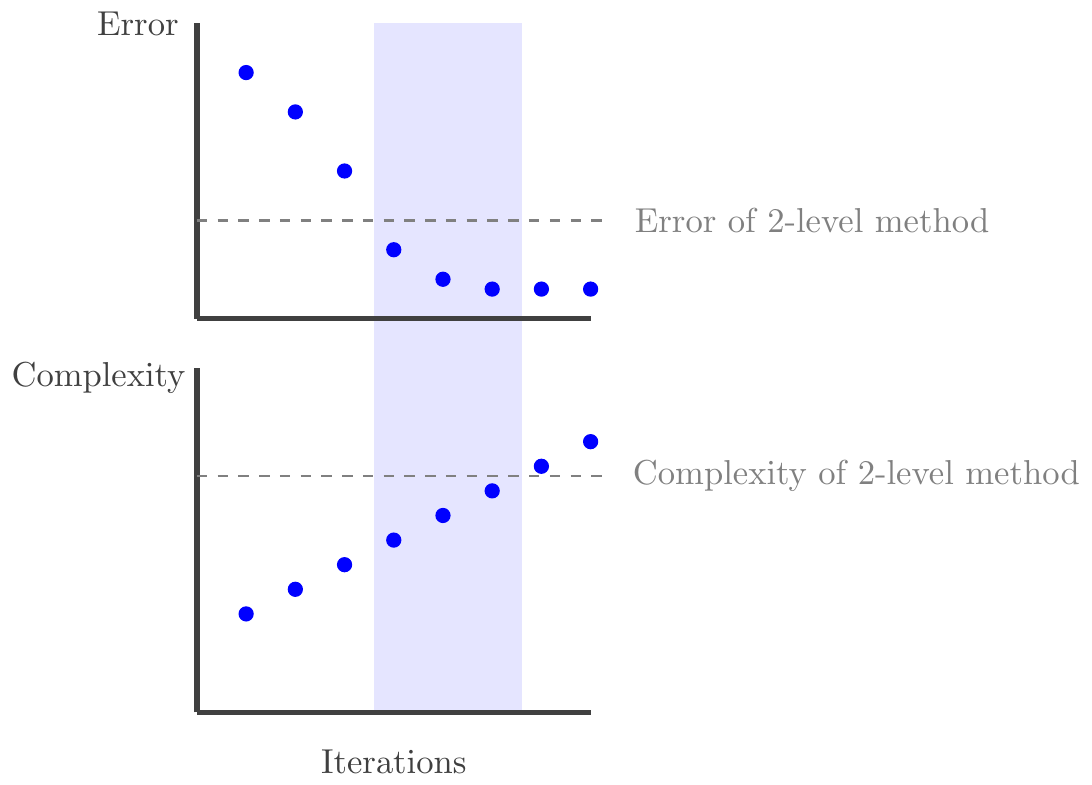}
    \caption{The dashed gray line in the upper figure illustrates the error of a two-level Parareal method. The blue dots in the upper figure show how the error of a Multi-level Parareal method decreases with an increasing number of iterations. The dashed gray line in the lower figure depicts the computational complexity of the two-level Parareal method. The blue dots in the lower figure illustrate how the computational complexity of the Multi-level Parareal method increases with an increasing number of iterations. The shaded blue area shows how many iterations can be done when the multi-level method shall be both more exact and less computationally complex than the two-level method.}
    \label{fig:complexity_error_multilevel_parareal}
\end{figure}

\subsection{Two-level Parareal and Two-level Parareal with averaging}
\label{subsec:Two-level_Parareal}

The two-level Parareal method has two levels, denoted as level 1 and level 0. Level 1 is the coarse level where we do $N$ time-steps on a coarse grid, which has $N+1$ grid points. Two neighboring grid points of the coarse grid form an interval. In total we have $N$ such small intervals and on each small interval we introduce a fine grid. These are the fine grids of level 0.

The Parareal method is a parallel-in-time method which has two basic solvers, a coarse propagator denoted as $G^1$ and a fine propagator denoted as $P^0$. The upper indices refer to the levels on which the propagators are applied. The coarse and the fine propagators can be Runge-Kutta methods, but other choices are possible too. First, the coarse propagator is applied on level 1 to compute initial guesses $U^0_n$, where the index $n$ counts the time-steps on level 1. The initial guesses are computed in serial and shall be improved iteratively. The values provided through the initial guess at the end of a coarse step are passed to the fine grids on level 0 as initial values. Then the fine propagator $P_0$ is applied in parallel. The results of the fine propagator are passed back to the coarse level 1. Applying the coarse propagator again, a Parareal iteration step can be computed 
\begin{equation}
    U_{n+1}^{1} =G^1(U_{n}^{1}) + P^0 (U_{n}^{0}) -G^1 (U_{n}^{0}) .
\end{equation}
The upper index of the numerical solutions counts the iterations. Once the first iteration is computed, it can be used as a new initial guess and the next iteration can be done.

This procedure is modified in the APinT method where averaging is incorporated, see \cite{Haut_Wingate_2014} and \cite{Peddle_Haut_Wingate_2019}. Here the coarse propagator, denoted as $\bar G^1$, provides a numerical solution to an averaged problem and not the original system \eqref{eq:modulation_equation}. Solving an averaged problem has the advantage that the fast oscillations, which are still in the modulation equation, are further mitigated, provided the averaging window $\eta$ is chosen appropriately. Thus, the oscillatory stiffness is mitigated and taking large time-steps is possible when the numerical method of the coarse propagator is applied. This can be beneficial for the efficiency of the algorithm. We can formulate a Parareal iteration step with averaging incorporated as follows
\begin{equation}
\label{eq:parareal_iteration_with_averaging}
    U_{n+1}^{k+1} =\bar G^1 (U_{n}^{k+1}) + P^0 (U_{n}^{k}) -\bar G^1(U_{n}^{k}) ,
\end{equation}
where $\bar G^1$ denotes the coarse propagator, which provides a numerical solution to the averaged problem on the coarse level, level 1.

In the averaged equation the right-hand side of \eqref{eq:modulation_equation} is replaced. In particular, when the coarse propagator is applied, an approximation to the following equation is computed
\begin{equation}
\label{eq:averaged_problem}
    \frac{d \bf \bar w}{dt} = \frac{1}{\eta} \int_{-\eta/2}^{\eta/2} \rho \left(\frac{s}{\eta} \right) \exp\left(\frac{L}{\epsilon} (s+t) \right) \mathcal{N} \left(\exp\left(-\frac{L}{\epsilon} (s+t) \right) {\bf \bar w}(t)\right) ds,
\end{equation}
where $\eta$ is the averaging window and $\rho$ is the kernel function. 
We use
\begin{equation}
    \rho (s) = \frac{1}{\rho_0} \exp \left({\frac{1}{(s-1/2)(s+1/2)} } \right),
\end{equation}
where $ \rho_0 $ normalizes the function.
In the right-hand side of \eqref{eq:averaged_problem} we write ${\bf \bar w}(t)$ to emphasize that the equation depends on $t$ only and not on $s$.
The oscillations in the equation with period $\le \eta$ are averaged. 
The exponential oscillates fast in $s$, whereas the non-linearity is slowly varying. The scaled filter function is slowly varying too, provided the averaging window $\eta$ is chosen large enough. In that case we integrate the oscillations which come from the exponential over a few periods. The length of the integration interval is not necessarily an integer multiple of the period and
damped oscillations remain after the integration. 
Averaged equations have been studied earlier, see for instance \cite{Sanders_etal_2007} or in the context of heterogeneous multi-scale methods \cite{E_Enquist_03} or \cite{Engquist_Tsai_2005}.


The algorithm with averaging is particularly promising for multi-scale problems. For each intrinsic scale of the problem we can introduce a level and  resolve the properties of the system which are the specific for that scale, i.e. every scale is assigned a level and the scale specific properties are resolved on the level assigned. The key besides the correct choice of the time-steps is to choose the averaging windows for the different levels such that the behavior on the finer scales is averaged but the  properties of the model on the scale corresponding to the level considered are still present.

\subsection{Multi-level Parareal and Multi-level Parareal with averaging}
\label{subsec:Two-level_Parareal_averaging}

Here we state again the idea behind the Multi-level Parareal algorithm: The Multi-level Parareal algorithm uses a recursion in the levels. Especially, the Multi-level Parareal algorithm with $L$ levels is a two-level Parareal algorithm where we compute a coarse solution with the coarse propagator on the coarsest level and the fine propagator is a Multi-level Parareal algorithm with $L-1$ levels applied on the finer levels. Then applying the Multi-level Parareal algorithm can be iterated until we reach the two-level case, where the two-level Parareal algorithm, described in the previous subsection, is applied. The multi-level algorithm with averaging included is illustrated in  \cref{fig:Multi_Level_Parareal}. However, we note here that it can be applied with and without averaging. When the method without averaging is applied, the coarse propagator solves the original system, for instance the modulation equation, and does not provide a solution to an averaged equation.

\begin{figure}[htbp]
  \centering
  \label{fig:a}
  \includegraphics[scale=0.85]{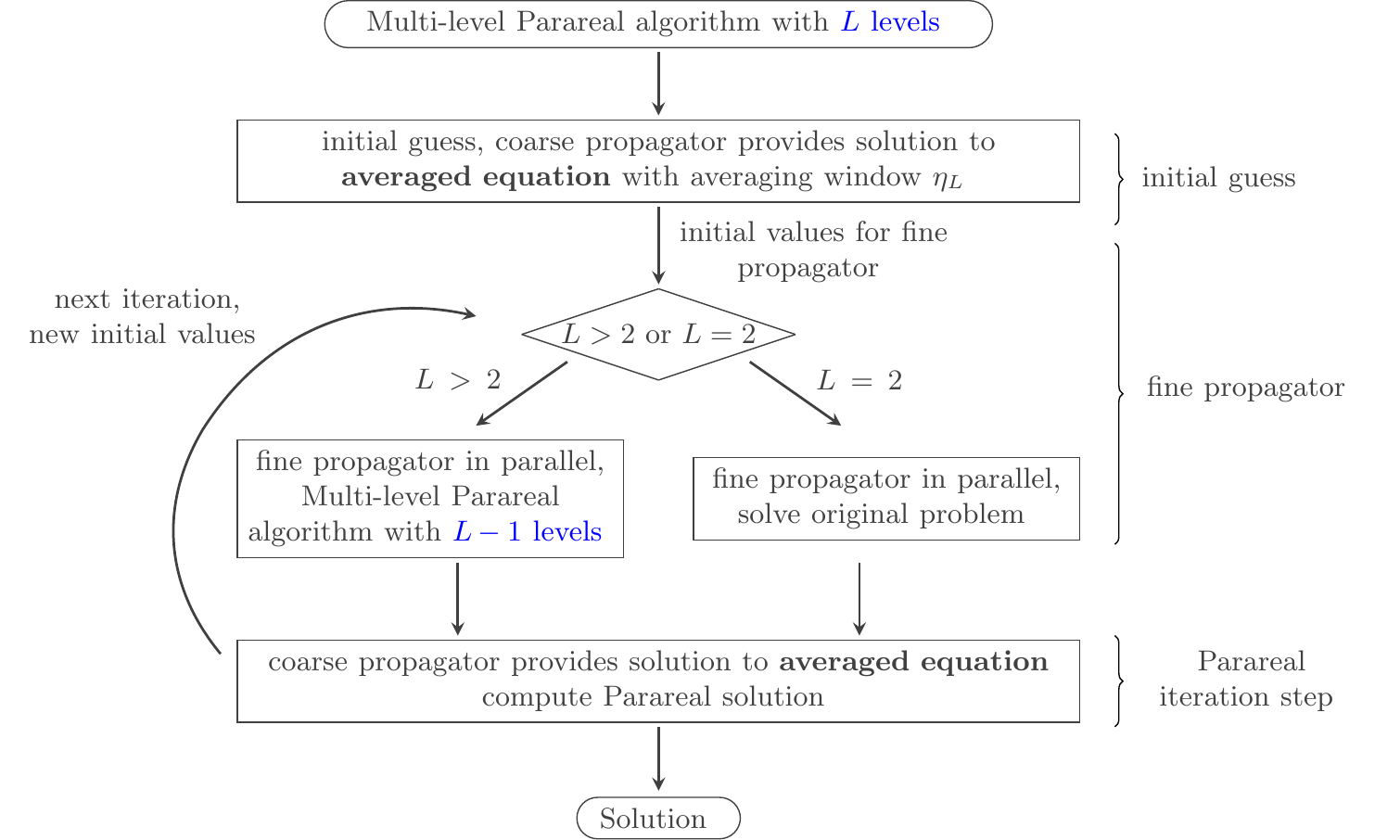}
  \caption{Multi-level Parareal algorithm with averaging}
  \label{fig:Multi_Level_Parareal}
\end{figure}

\begin{figure}
    \centering
    \includegraphics{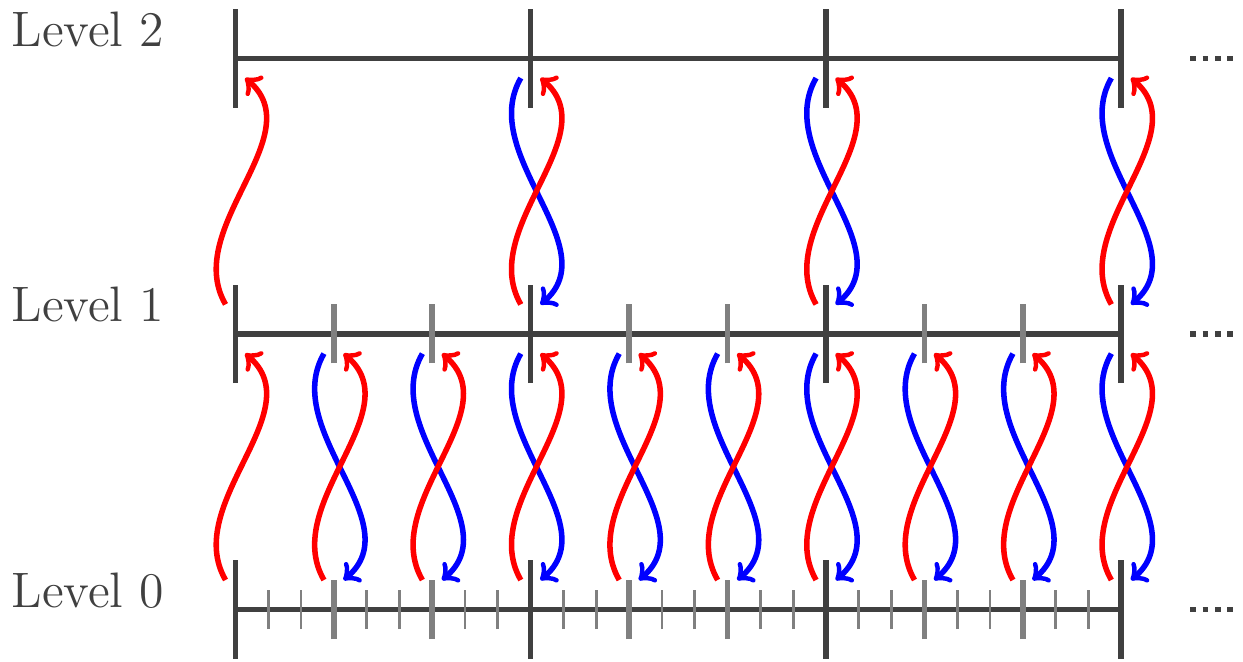}
    \caption{Three levels with coarsening factor $N=3$. Blue arrows indicate how coarse solutions, computed on level $l$, are passed from  level $l$ to level $l-1$. On level $l-1$ the coarse solutions are used as initial values for the propagator on that level. Red arrows indicate  how fine solutions, computed with the propagator on level $l-1$, are passed to level $l$ where the Parareal iteration step is done using the fine solutions.}
    \label{fig:MultiLevelParareal_InitialValues_FineSolution}
\end{figure}

To distinguish between the different levels, we introduce the subsequent notation: When we consider a Multi-level Parareal method with $L$ levels, level $L-1$ denotes the coarsest level and level $0$ is finest level. If $l_1 > l_2$ for $l_1, l_2 \in \{ 0, \dots L-1 \}$, then level $l_1$ is the coarser and level $l_2$ the finer level. The coarse propagator applied on level $l$ is denoted by $G^l$ or in the case when an averaged equation is solved by $\bar G^l$. The fine propagator needed to compute the Parareal solution on level $l$ is denoted as $P^{l-1}$ or $\bar P^{l-1}$. In the rest of the subsection we will write $\bar G^l$ or $\bar P^{l-1}$. The case without averaging can be concluded straightforwardly.

Let us now explain one iteration of the Multi-level Parareal method with $L$ levels, assuming that the fine propagators, $\bar P^{L-2}, \dots,\bar  P^1$, are  Parareal methods with only one iteration too. On the coarsest level, level $L-1$, the coarse propagator, $\bar G^{L-1}$, provides the initial guess, which is the numerical solution to an averaged problem of the form \eqref{eq:averaged_problem}, however with an indexed averaging window $\eta_{L-1}$. This solution shall now be improved iteratively according to relation \eqref{eq:parareal_iteration_with_averaging}. This means we have to compute fine solutions. The fine solutions on level $L-2$ are computed in parallel and the coarser solution from level $L-1$ gives the initial values. In the case of $L-2 =0$ the fine propagator is a basic numerical scheme, like a Runge-Kutta method, and not a Parareal method. In the other case, $L-2 > 0$, the fine propagator is a Parareal method with $L-1$ levels. This Parareal method with $L-1$ levels again has a coarse and a fine propagator. The coarse propagator of that method is applied on the level $L-2$ and is a basic time-stepping scheme, like a Runge-Kutta method. Although denoted as coarse, the coarse propagator applied on the level $L-2$ uses finer time-steps than the coarse propagator on level $L-1$. Moreover, the coarse propagator applied on the level $L-2$ computes a numerical solution to an averaged equation again of the form \eqref{eq:averaged_problem} but this time with averaging window $\eta_{L-2}$. As level $L-2$ is a finer level with finer grids than level $L-1$, we want to resolve finer scales on level $L-2$. This means that we resolve some of the oscillations on level $L-2$ that were averaged on level $L-1$. Accordingly, we have $\eta_{L-2} < \eta_{L-1}$. Consequently, oscillations with period $>O(\eta_{L-2})$ are resolved. We continue the recursion in the levels until we reach level $0$. Only on level 0 the unaveraged equation \eqref{eq:modulation_equation} is solved. Once a fine solution is computed, we can apply a Parareal iteration step according to equation \eqref{eq:parareal_iteration_with_averaging}. Especially, level 0 provides the fine solutions for level 1, needed to apply equation \eqref{eq:parareal_iteration_with_averaging} to compute a Parareal solution on level 1. Then, level 1 gives the fine solutions for level 2 and so on. We can continue this procedure until we reach the coarsest level, level $L-1$. \cref{fig:MultiLevelParareal_InitialValues_FineSolution} illustrates how the coarse solutions are passed as initial values to the propagators on the finer levels and the solutions on the finer levels are returned to the coarser levels so that the Parareal iteration step can be done.

A basic principle of the Parareal method is that an initial guess is iteratively improved. These correction iterations are done on the levels $L-1, L-2, \dots, 1$ in the multi-level case,
 and inspired by multi-grid methods we can adopt the terminology of cycles. Illustrations of examples with three levels can be found in  \cref{fig:V_Cycle} and \cref{fig:W_Cycle}.  \cref{fig:V_Cycle} shows the case where one iteration is done on the levels 1 and 2, in \cref{fig:W_Cycle} we do two iterations on the levels 1 and 2.

\begin{figure}[htbp]
  \centering
 \includegraphics{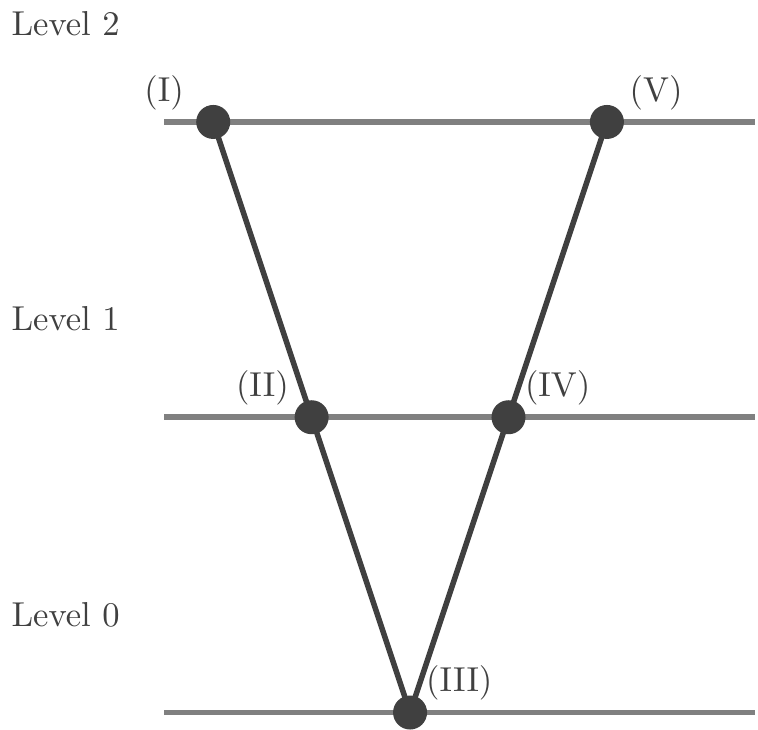}
  \label{fig:V_Cycle}
  \caption{Parareal algorithm with 3 levels, $k_2=1$ iteration on Level 2 and $k_1=1$ iteration on Level 1, (I) Compute initial guess on coarsest level, Level 2, (II) get initial values from Level 2, compute initial guess on Level 1 in parallel, (III) apply fine propagator in parallel, get initial values from Level 1, (IV) do Parareal iteration step with fine solution from Level 0, (V) do Parareal iteration step with fine solution from Level 1}
\end{figure}

\begin{figure}[htbp]
  \centering
 \includegraphics[scale=0.75]{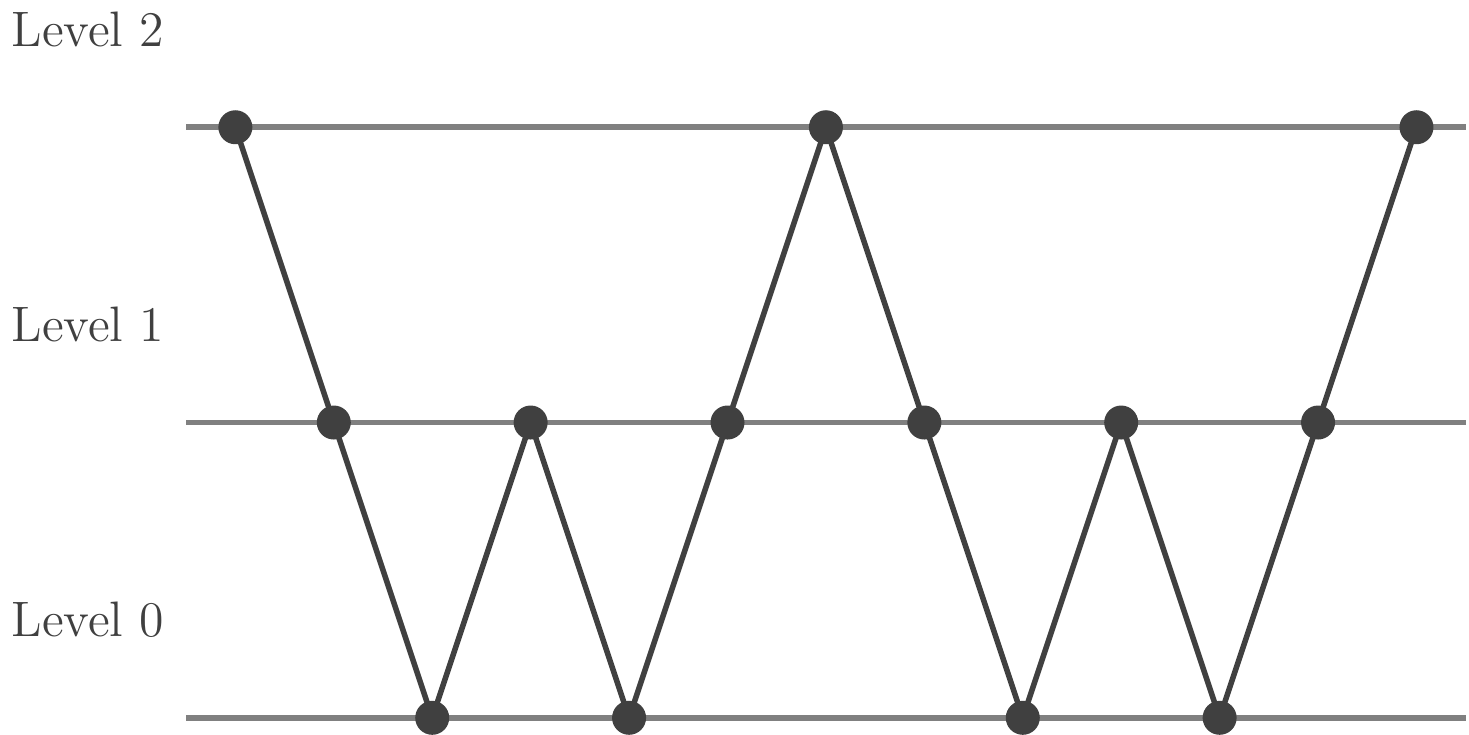}
 \label{fig:W_Cycle}
 \caption{Parareal algorithm with 3 levels with $k_2=2$ iterations on Level 2 and $k_1=2$ iterations on Level 1, Level $l$ provides the initial values for Level $l-1$, on Level $l-1$ we compute the fine solution with the fine propagator for Level $l$, where we do the Parareal iteration step }
\end{figure}


\section{Convergence results}
\label{sec:Asymptotic_convergence}


Here we state convergence results for two-level Parareal methods which are known from the literature. The stated results hold when the fine propagator computes the exact solution.

The first result can be interpreted in a dissipative context. In that case the problem, which shall be solved, has no fast oscillations, especially there is no time-step constraint due to fast oscillations.
The authors in \cite{Gander_Hairer_2008} introduce the following conditions on the coarse propagator $G^1$, which solves an unaveraged equation:
\begin{equation}
\label{eq:(7)_Gander_Hairer}
    E(x) - G^1(x) = c_{p+1}(x) \Delta T^{p+1} + c_{p+2} (x) \Delta T^{p+2} + \dots,
\end{equation}
where $E$ denotes the exact propagator of the unaveraged problem over the time horizon $\Delta T$, and
\begin{equation}
\label{eq:(8)_Gander_Hairer}
    \|G^1(x) - G^1(y)\| \le (1+C_2\Delta T) \|x - y\| .
\end{equation}
Moreover, they give the following result:
\begin{theorem}
\label{thm:standard_two_level_estimate}
Let the error of the coarse propagator $G^1$ be bounded by the truncation error $C_1 \Delta T^{p_1+1}$. Moreover, suppose $G^1$ satisfies \eqref{eq:(7)_Gander_Hairer} and the Lipschitz condition \eqref{eq:(8)_Gander_Hairer}. Then the Pararael solution satisfies
\begin{equation}
    \|y(t_n) -Y_n^k \| \le \frac{C_3C_1^kT_n^{k+1}}{(k+1)!} \exp(C_2(T_n-T_{k+1})) \left( \Delta T \right)^{p_1(k_1+1)} ,
  \end{equation}
  where  $p_1$ is accuracy order of
  coarse propagator and $k_1$ is number of iterations.
\end{theorem}

Another result with averaging incorporated is given in \cite{Peddle_Haut_Wingate_2019}. The work done in \cite{Peddle_Haut_Wingate_2019} has a focus on dealing with fast oscillations which appear due to a skew Hermitian linear operator and cause time-step constraints. These time-step constraints are circumvented by combining the coarse propagator with an averaging technique, which mitigates the oscillatory stiffness.
\begin{theorem}
Suppose the coarse propagator satisfies the same conditions as in Theorem \ref{thm:standard_two_level_estimate}. (Note: The coarse propagator solves the averaged equation \eqref{eq:averaged_problem}, in particular we consider the truncation error when the averaged problem is solved.)
Let the nonlinearity $\mathcal{N}$ satisfy the Lipschitz condition
\begin{equation}
    \max \|\mathcal{N}(t, v(\tau_1)) - \mathcal{N}(t, v(\tau_2)) \| \le \lambda \|v(\tau_1) - v(\tau_2) \|
\end{equation}
and 
\begin{equation}
    M = \max \| \mathcal{N}(t, v(\tau))\| < \infty.
\end{equation}
Then, the error after the $k_1$ Parareal iterations can be bounded by
\begin{equation}
    \|y(t_n) -Y_n^k \| \le \frac{C^{k_1+1}}{(k_1+1)!} \exp(C) (\varepsilon \eta +  \kappa \Delta
    T_{}^{p_{1}} ) \left( \frac{\varepsilon \eta}{\Delta
    T_{}} + \kappa \Delta T_{}^{p_{1}}\right)^{k_1} .
\end{equation}
\end{theorem}

\subsection{Multi-level result without averaging}
\label{subsec:Asymptotic_convergence_without_averaging}

The generalization of the result given in \cref{thm:standard_two_level_estimate} to the multi-level case has two important building blocks. First, a bound must be found when the fine propagator is not exact. This is necessary, as the fine propagator of the Multi-level Parareal algorithm is a Multi-level Parareal method (with one level less) and we want to investigate on which levels numerical errors emerge and how they develop as the fine solutions are passed to the coarser levels. Second, an inductive argument which takes the errors on all levels into account will be applied.

\subsubsection{A two-level result with non-exact fine propagator}
\label{subsec:two-level_result_with_non-exact_fine_propagator}

The first step is to modify the result given in \cite{Gander_Hairer_2008} to the case when the fine propagator is not exact. 
The time-step on level 0 is denoted as  $\Delta T_0$ and on level 1 as  $\Delta T_1$.

\begin{theorem}
\label{thm:two-level_bound}
We assume the same regularity conditions on $G^1$ as in  \cref{thm:standard_two_level_estimate}. Additionally, assume that the truncation error of the fine propagator $P^0$ is bounded by $c \Delta T_0^{p_0+1}$. Then, the error of the two-level Parareal algorithm with non-exact fine propagator,$P^0$, can be bounded by the following estimate
\begin{equation}
\label{eq:non-exact_fine_propagator}
\begin{split}
    \| u(T_n) - U_n^{k+1} \| & \le {n \choose k+1} (C_3 \Delta T_1^{p_1+1}) (C_1 \Delta T_1^{p_1+1})^k  (1+C_2 \Delta T_1)^{n-k-1} \\
     & \quad + (c  T \Delta T_0^{p_0}) (1+C_2 \Delta T_1)^{n-1} (1+(C_1 \Delta T_1^{p_1+1}))^{n-1} \\
     & \le \frac{C_3C_1^kT_n^{k+1}}{(k+1)!} \exp(C_2(T_n-T_{k+1})) \left( \Delta T_1 \right)^{p_1(k_1+1)} \\
     &\quad + \exp\bigg(C_2 \left(T_n-T_1 \right)+C_1 \Delta T_1^{p_1} \left(T_n-T_1\right)\bigg) c  T \Delta T_0^{p_0} .
     \end{split}
\end{equation}
\end{theorem}

The grid points of the coarse time grid are donted as $T_0, T_1, \dots, T_N$ in the theorem. Both the fine propagator and the coarse propagator solve the unaveraged problem. The bound in \cref{eq:non-exact_fine_propagator} is the two-level estimate from \cref{thm:standard_two_level_estimate} with an additional term depending on the accuracy of the fine propagator.

\begin{proof}

For the following equation, which describes one Parareal correction iteration, an estimate must be found
\begin{equation*}
    u(T_n) - U_n^{k+1}  = E(u(T_{n-1})) - G^1(U_{n-1}^{k+1}) - P^0(U_{n-1}^{k}) +
    G^1(U_{n-1}^{k}) ,
\end{equation*}
where $E$ denotes the exact, $G^1$ the coarse and $P^0$ the fine propagators.
The equation can be rewritten as
\begin{align*}
 u(T_n) - U_n^{k+1}
    & = E(u(T_{n-1})) - G^1(u(T_{n-1})) \\
    & \quad + G^1(U_{n-1}^{k}) - E(U_{n-1}^k) \\
    & \quad +G^1(u(T_{n-1}))  -G^1(U_{n-1}^{k+1})\\
    & \quad + E(U_{n-1}^k)- P^0(U_{n-1}^{k}) .
\end{align*} 
 To bound the first two lines relation \eqref{eq:(7)_Gander_Hairer} can be applied. The third line can be estimated by applying the inequality \eqref{eq:(8)_Gander_Hairer}. The last line measures the accuracy of the fine propagator 
 \begin{equation*}
     \| E(U_{n-1}^k)- F(U_{n-1}^{k}) \|
      \le c \Delta T_1 \Delta T_0^{p_0},
 \end{equation*} 
where $ \Delta T_0$ denotes the fine time-step and $p_0$ is the order of accuracy of the fine propagator. The error constant of time-stepping methods, like Runge-Kutta methods, contains the length of the integration interval as a factor. That is why the coarse timestep $\Delta T_1$ can be found in the above estimate.

Consequently, we arrive at the following estimate
\begin{equation*}
\begin{split}
    & \|u(T_n) - U_n^{k+1} \| \\ 
    & \qquad \le C_1 \Delta T_1^{p_c+1} \| u(T_{n-1})- U_{n-1}^k \|
    + (1+C_2 \Delta T_1) \| u(T_{n-1})- U_{n-1}^{k+1} \|
    + c \Delta T_1 \Delta T_0^{p_0}.
    \end{split}
\end{equation*}
Compared to earlier results (see for example the proof in \cite{Gander_Hairer_2008}), we have an additional term which comes from the fine propagator. Moreover, $ \Delta T_1$ denotes the coarse time step and $p_1$ is the order of accuracy of the coarse propagator.

The estimate motivates to consider the following recurrence relation
\begin{equation}
\label{eq:recursion_relation}
    e_n^{k+1} = \alpha e_{n-1}^k + \beta e_{n-1}^{k+1} + \delta_n, \qquad \qquad
    e_n^0 = \gamma + \beta e_{n-1}^0 ,
\end{equation}
with $\delta_0 =0$ and $\delta_n = \delta$ for $n \ge 1$. In addition, we set $\alpha = C_1 \Delta T_1^{p_1+1}$,
$\beta = (1+C_2 \Delta T_1)$,
$\delta = c \Delta T_1 \Delta T_0^{p_0}$ and
$\gamma = C_3 \Delta T_1^{p_1+1}$.

Remark about the $e_n^0$: The coarse propagator is used to compute the initial guess. For the error of the initial guess the following relations hold:
\begin{align*}
    e_n^0 =\| u(t_n) -U_n^0 \| &= \|E(u(t_{n-1})) - G^1(U_{n-1}^0) \| \\
    & \le \|E(u(t_{n-1})) - G^1(u(t_{n-1})) \| + \| G^1(u(t_{n-1})) - G^1(U_{n-1}^0)\| \\
    & \le C_3 \Delta T_1^{p_1+1} + (1+C_2 \Delta T_1) \| G^1(u(t_{n-1})) - G^1(U_{n-1}^0)\| \\
    & = C_3 \Delta T_1^{p_1+1} + (1+C_2 \Delta T_1) \| e_{n-1}^0 ,
\end{align*}
where the relations \eqref{eq:(7)_Gander_Hairer} and \eqref{eq:(8)_Gander_Hairer} were applied.

An estimate for $e_n^k$ is given in 
Lemma \ref{lemma:reucrsion_equation_1} through
\begin{align*}
     e_n^k & \le {n \choose k+1}    \gamma \alpha^k  \beta^{n-k-1} + n \delta \beta^{n-1} (1+\alpha_{0n})^{n-1} \\
     & \le {n \choose k+1} (C_3 \Delta T_1^{p_1+1}) (C_1 \Delta T_1^{p_1+1})^k  (1+C_2 \Delta T_1)^{n-k-1} \\
     & \quad + n (c \Delta T_1 \Delta T_0^{p_0}) (1+C_2 \Delta T_1)^{n-1} (1+(C_1 \Delta T_1^{p_1+1}))^{n-1} \\
     & \le {n \choose k+1} (C_3 \Delta T_1^{p_1+1}) (C_1 \Delta T_1^{p_1+1})^k  (1+C_2 \Delta T_1)^{n-k-1} \\
     & \quad + (c T \Delta T_0^{p_0}) (1+C_2 \Delta T_1)^{n-1} (1+(C_1 \Delta T_1^{p_1+1}))^{n-1} .
\end{align*}
For the last inequality, we exploited that $n \Delta T_1 \le  T$, where $[0,T]$ is the interval over which we want to solve the differential equation. Especially, we have $N \Delta T_1 =  T$, where $N$ denotes the number of coarse intervals. \\

{\color{white}1}
\end{proof}

\subsubsection{Generalization to multiple levels}
\label{subsec:multi-level_result}

The next step is to investigate the error estimate for the Multi-level Parareal method, and in particular how the errors introduced on the different levels influence the total error. We assume that we are given $L$ levels, where level 0 corresponds to the finest level and level $L-1$ is the level with coarsest grid. 

The proof for the Multi-Level method will be an inductive proof where we can apply similar arguments as in the two-level case. The difference is that the fine propagator is not a basic ODE solver, like a Runge-Kutta scheme, but it is the Multi-Level Parareal method on the finer grid. Only on the finest level, level 0, we apply a basic ODE solver. In particular, this changes the appearance of $\delta$ in the proof.

The size of the time-step is different for different levels. Therefore, $\alpha, \beta, \gamma$ and $\delta$, which are defined in the proof in the previous \cref{subsec:two-level_result_with_non-exact_fine_propagator}, should have an index $l$. More precisely, we write $\alpha_l, \beta_l, \gamma_l$ and $\delta_l$ instead of $\alpha, \beta, \gamma$ and $\delta$. In addition, the number of Parareal iterations does not have to be the same for each level. Consequently, $k$ should also have an index $l$ and we should write $k_l$.

Moreover, we introduce a coarsening factor $N$ which relates the time-steps on the different levels. Let $ \Delta T_0$ be the time-step on the finest grids on level 0. Then, the next finest grids are on level 1 and the time-step is given by $\Delta T_1 = N \Delta T_0$. For the next coarsest level, level 2, we have a time-step $\Delta T_2 = N \Delta T_1= N^2 \Delta T_0$ and so on. See also \cref{subsec:lemmata_multiple_levels} and especially the relations \cref{eq:relation_for_Delta_T_l_1} and \cref{eq:relation_for_Delta_T_l_2}.

\begin{theorem}
\label{thm:multi-level_result}
Let the propagators for the levels $1, \dots, L-1$ satisfy the conditions on the coarse propagator in \cref{thm:standard_two_level_estimate}.  Suppose that the fine propagator, on level 0, satisfies the condition on the fine propagator from \cref{thm:two-level_bound}.
Then we can show that the following error estimate for the Multi-level case holds
\begin{equation}
    e_{n,l}^{k_l} \le \sum_{\bar l =1}^l E_{ \bar l} \prod_{j= \bar l +1}^l A_{j} + \delta_0 \prod_{ \bar l = 1 }^l A_{ \bar l},
\end{equation}
where 
\begin{align} 
    E_l &= {n \choose k_l+1}    \gamma_l \alpha^{k_l}_l  \beta^{N-{k_l}-1}_l  \label{eq:E_l__} \\
    \delta_0 &= c \Delta T_1 \Delta T_0^{p_0} \label{eq:delta_0}\\
    A_l &=  N \beta^{N-1}_l  (1+\alpha_{0N,l})^{N-1}. \label{eq:A_l__} 
\end{align}
We have $\alpha_l = C_{1,l} \Delta T_l^{p_l+1} , 
\beta_l = (1+C_{2,l} \Delta T_l), \gamma_l =  C_{3,l} \Delta T_l^{p_l+1} \text{ and } \alpha_{0N,l} \in [0,\alpha_l] $ .
\end{theorem}

Especially, the derived error estimate is a sum. The different terms of the sum are composed of the error $E_l$ emerging on  level $l$ and amplification factors $A_l$ which amplify the errors made on the finer levels as the fine solutions are passed to the coarser levels.

\begin{proof}
As the Multi-level method is a special case of the two-level method, just as in  \cref{thm:two-level_bound} we arrive at the recurrence relation
\begin{equation*}
    e_n^{k_l+1} = \alpha_l e_{n-1}^{k_l} + \beta_l e_{n-1}^{k_l+1} + \delta_{l-1}, \qquad \qquad
    e_n^0 = \gamma_l + \beta_l e_{n-1}^0 .
\end{equation*}
The only difference is that $\delta_{l-1}$ admits a different form. Particularly, $\delta_{l-1}$ is the error bound of a Multi-level Parareal method with one level less. Like in the proof of \cref{thm:two-level_bound} we find that
\begin{equation*}
     e_{n,l}^{k_l} \le {n \choose k_l+1}    \gamma_l \alpha^{k_l}_l  \beta^{n-{k_l}-1}_l + n \delta_{l-1} \beta^{n-1}_l (1+\alpha_{0n,l})^{n-1}
\end{equation*}
for $k<n$.
The index $l$ denotes the level. For the finer levels we have to assume that $n=N$ in the error estimate, because we have to to do the time-stepping completely through the fine grid to provide the fine solution for the next coarser level. We can rewrite the above estimate as
\begin{equation*}
   e_{n,l}^{k_l} \le E_{l} + A_{l} \delta_{l-1} ,
\end{equation*}
where $\delta_{l-1} = e_{N,l-1}^{k_{l-1}}$. Additionally, we have introduced $E_{l} = {n \choose k_l+1}    \gamma_l \alpha^{k_l}_l  \beta^{N-{k_l}-1}_l$ (relation \eqref{eq:E_l}) and 
$ A_{l} = N \beta^{N-1}_l  (1+\alpha_{0N,l})^{N-1}$ (relation \eqref{eq:A_l}). This estimate holds for any $n \leq  N$. According to Lemma \ref{lemma:first_estimate_theorem2} we have
\begin{equation*}
    e_{n,l}^{k_l} \le \sum_{\bar l =1}^l E_{ \bar l} \prod_{j= \bar l +1}^l A_{j} + \delta_0 \prod_{ j = 1 }^l A_{ j}.
\end{equation*}
The Proof of Lemma \ref{lemma:first_estimate_theorem2} contains the {\bf inductive argument} announced earlier.

{\color{white} 1}
\end{proof}

\begin{corollary}
Assuming that a constant coarsening factor $N$ relates the different levels and the number of iterations on the levels is constant, i.e. $k_l = k$ for all $l=1,\dots,L-1$, we get the following error bound

\begin{equation}
\begin{split}
    \| u(T_n) - U_n^{k+1} \| 
    & \le  cT { \Delta T_0^{p_0}} \exp \left(C_2 T \frac{1-1/N^L}{1-1/N} +
      C_1 T \Delta T_1^{p_c}
      \frac{1-1/N^{L(p_c+1)}}{1-1/N^{p_c+1}}\right) + \\
    & \hspace{0.4cm} \dots + \exp \left(\frac{C_2 T}{1-1/N} +
      \frac{C_1 T \Delta T_1}{1-1/N^{p_c+1}} \right) \dots \\ &
    \hspace{1.7cm}
   C_3 C_1^{k} {N \choose k +1} 
    \frac{1}{1-1/N^{k p_c +k+p_c}} {\Delta T_1^{k p_c + k + p_c+1}}
\end{split}
\end{equation}
(The notational conventions from \cref{subsec:lemmata_multiple_levels} are used) Particularly, we recover the accuracy order of the two level scheme.
\end{corollary} 
 
\begin{proof}
The corollary is an immediate consequence of the results given in \cref{thm:multi-level_result}, lemma \ref{lemma:A5} and lemma \ref{lemma:A6}.
{\color{white} 1}
\end{proof}

\subsection{Multi-level result with averaging}
\label{subsec:Asymptotic_convergence_with_averaging}

In this section an 
convergence result for the Multi-level Parareal algorithm with averaging shall be presented.  Assuming we have $L$ levels in total, the averaging is done on the levels $1, \dots, L-1$. Only on level 0, the level with the finest grids, the full, unaveraged system is solved by the finest propagator.

Especially, the finest propagator solves the following exact equation
\begin{equation}
\label{eq:unaveraged_equation}
\frac{du}{dt}(t) + e^{t/\epsilon L}\mathcal{N}(e^{-t/\epsilon L} u(t)) = 0.    
\end{equation}

The coarse propagators do not solve the exact equation. Instead, they solve averaged equations, given through 
\begin{equation}
\label{eq:averaged_equation}
\frac{d \bar u(t)}{dt}(t) +  e^{t/\epsilon L} \frac{1}{\eta_l} \int_{-\eta_l/2}^{\eta_l/2} e^{s/\epsilon L}\mathcal{N}(e^{-s/\epsilon L} \bar u(t)) \ ds = 0 .    
\end{equation}
This is the crucial difference between the Parareal algorithms with and without averaging.

We will use the following notation to denote the propagators relevant for this section:

\begin{itemize}
\item $E$ - analytically exact solver for the unaveraged equation \eqref{eq:unaveraged_equation}
\item $\bar E^l$ - analytically exact solver for the averaged equation \eqref{eq:averaged_equation}
\item $\bar G^l$ - numerical solver of the averaged problem \eqref{eq:averaged_problem} on level $l$, $l= 1, \dots, L-1$, also the coarse propagator for $\bar P^l$
\item $\bar P^l$ - Multi-level Parareal method with averaging with the coarse propagator $\bar G_l$ and  fine propagator $\bar P^{l-1}$ for $l>1$ (or $P^0$ for $l=1$)
\item $P^0$ - numerical solution of the unaveraged problem \eqref{eq:unaveraged_equation} on level $0$ 
\end{itemize}

For the propagators $\bar G^l$, $l= 1, \dots, L-1$, we impose the following conditions on the truncation error
\begin{equation}
\label{eq:(7)_Gander_Hairer_modified}
    E(x) - \bar G^l(x) = \bar c^{\ l}_{p+1}(x) \Delta T_l^{p+1} + \bar c^{\ l}_{p+2} (x) \Delta T_l^{p+2} + \dots, \quad \text{ and }
\end{equation}
 \begin{equation}
\label{eq:(8)_Gander_Hairer_modified}
    \|\bar G^l(x) - \bar G^l(y)\| \le (1+C^l_2\Delta T_l) \|x - y\| .
\end{equation}

\begin{theorem}
\label{theorem:Multi-level_result_with_averaging}
Let the error of the coarse propagators $\bar G^l$ be bounded by the truncation error $\bar C^l_1 \Delta T_l^{p_1+1}$. Moreover, suppose $\bar G^l$ satisfies relation \eqref{eq:(7)_Gander_Hairer_modified} and the Lipschitz condition \eqref{eq:(8)_Gander_Hairer_modified} for $l= 1, \dots, L-1$.
  Let the fine propagator, on level 0, satisfy the condition on the fine propagator from \cref{thm:two-level_bound}.
Then we can show that the following error estimate for the Multi-level Parareal method with averaging holds
\begin{equation}
    e_{n,l}^{k_l} \le \sum_{\bar l =1}^l \bar E_{ \bar l} \prod_{j= \bar l +1}^l A_{j} + \delta_0 \prod_{ \bar l = 1 }^l A_{ \bar l},
\end{equation}
where 
\begin{align} 
    \bar E_l &= {n \choose k_l+1} \bar \gamma_l \bar \alpha^{k_l}_l  \beta^{N-{k_l}-1}_l  \label{eq:E_l} \\
    \delta_0 &= c \Delta T_1 \Delta t^{p_0} \label{eq:delta_0}\\
    A_l &=  N \beta^{N-1}_l  (1+\bar \alpha_{0N,l})^{N-1}. \label{eq:A_l} 
\end{align}
We have 
\begin{align*}
    &\bar \alpha_l = C \eta_l \epsilon + C \Delta T_l^{p_c+1} \max_{\omega_0 \le \omega} \left\{ \left|\frac{\epsilon}{\omega} \right| \kappa(\epsilon,\eta_l, \omega) \right\} \\
    &\beta_l = (1+C_{2,l} \Delta T_l) \\
    &\bar \gamma_l =  C \eta_l \epsilon \|\mathcal{\tilde M}_1\| + C \Delta T_l^{p_c+1} \max_{\omega_0 \le \omega} \left\{ \left|\frac{\epsilon}{\omega} \right| \kappa(\epsilon,\eta_l, \omega) \right\} \|\mathcal{\tilde M}_0\|  \text{ and } \\
    &\alpha_{0N,l} \in [0,\bar \alpha_l]  .
\end{align*}

\end{theorem}

Remarks: 
\begin{enumerate}
    \item The error constants which arise as numerical time stepping schemes are applied to the differential equations specific to the different levels should be level dependent. For stiff problems the constants are big, for non-stiff problems they are smaller. Consequently, when the averaging windows are chosen appropriately, the constants are roughly of the same order of magnitude for each level.
    \item The structure of the amplification factor is the same as in the case without averaging. However, the terms $\bar \alpha_l$ and $\bar \gamma_l$, which cause the contraction in the error, have an additional term that accounts for the error due to averaging.
\end{enumerate}

\begin{proof}

Again the proof has two building blocks. First, we build upon the two-level proof in \cite{Peddle_Haut_Wingate_2019}, but additionally assume that the fine propagator is not exact. Second, an inductive argument is applied to obtain a multi-level result.

The averaging procedure introduces new errors which must be accounted for in the error analysis. We can start with the same arguments as in \cref{thm:two-level_bound},
however we need new estimates for the lines in the following equation
\begin{align*}
 u(T_n) - U_n^{k+1}
    & = E(u(T_{n-1})) - \bar G^1(u(T_{n-1})) \\
    & \quad + \bar G^1(U_{n-1}^{k}) - E(U_{n-1}^k) \\
    & \quad +\bar G^1(u(T_{n-1}))  -\bar G^1(U_{n-1}^{k+1})\\
    & \quad + E(U_{n-1}^k)- P^0(U_{n-1}^{k}) .
\end{align*}

The estimate for the last line remains the same as in the proof without averaging and the accuracy estimate  of the fine propagator is given by
\begin{equation}
    \|  E(U_{n-1}^k)- P^0(U_{n-1}^{k}) \| \le c \Delta T_1 \Delta T_0^{p_0}.
\end{equation}

The first two lines can be rewritten as
\begin{align*}
 & E(u(T_{n-1})) - \bar G^1(u(T_{n-1})) + \bar G^1(U_{n-1}^{k}) - E(U_{n-1}^k) \\  & \qquad \qquad \qquad \qquad =
 E(u(T_{n-1})) - \bar E^1 (u(T_{n-1}))+ \bar E^1 (u(T_{n-1}))  - \bar G^1(u(T_{n-1})) + \dots \\ 
 & \qquad \qquad \qquad \qquad \quad +  P^1(U_{n-1}^{k}) - \bar E^1 (U_{n-1}^k)+ \bar E^1 (U_{n-1}^k)  - E(U_{n-1}^k)\\
 &  \qquad \qquad \qquad \qquad = \mathcal{M}_{1,1}(u(T_{n-1}), \epsilon, \eta_1) + \mathcal{M}_{0,1}(u(T_{n-1}), \epsilon, \eta_1, \Delta T) \dots \\ 
  & \qquad \qquad \qquad \qquad \quad   - \mathcal{M}_{1,1}(U_{n-1}^k, \epsilon, \eta_1) - \mathcal{M}_{0,1}(U_{n-1}^k, \epsilon, \eta_1, \Delta T) , 
\end{align*}
see the relations \eqref{eq:definitions_of_M0_and_M1} for the definitions of $ \mathcal{M}_{0,1}$ ans $ \mathcal{M}_{1,1}$.
The function $\mathcal{M}_{0,l}(v, \epsilon, \eta_l, \Delta T) $, where the index $l$ refers to the level, has a representation of the form
\begin{equation}
    \mathcal{M}_{0,l}(v, \epsilon, \eta_l, \Delta T) \le C \Delta T^{P_c+1} \max_{\omega_0 \le \omega} \left\{ \left|\frac{\epsilon}{\omega} \right| \kappa(\epsilon,\eta_l, \omega) \right\} \mathcal{\tilde M}_{0}(v) \ ,  
\end{equation}
 see \cite{Peddle_Haut_Wingate_2019}. Moreover, $\mathcal{M}_{0,l}(v, \epsilon, \eta_l, \Delta T)$ is Lipschitz continuous in the first component.

In addition, for the function $\mathcal{M}_{1,l}(v, \epsilon, \eta) $ a relation of the form
\begin{equation}
\label{eq_M_1l_est}
    \mathcal{M}_{1,l}(v, \epsilon, \eta_l) \le C \eta_l \epsilon \mathcal{\tilde M}_{1}(v) \,
\end{equation}
can be found where $\mathcal{M}_{1}(v, \epsilon, \eta) $ is Lipschitz continuous in $v$ , see \cite{Peddle_Haut_Wingate_2019} or lemma \ref{lemma:Lemma_about_M1}.
We can take $ \mathcal{\tilde M}_{1}$ and $ \mathcal{\tilde M}_{0}$ as the maximum of $ \mathcal{\tilde M}_{1,l}$ and $\mathcal{\tilde M}_{0,l}$ over all levels $l$.

We use the subsequent relations which are given in \cite{Peddle_Haut_Wingate_2019}
\begin{align}
    \| \bar G^l(u(T_{n-1})) -\bar G^l(U_{n-1}^{k+1})\| &\le (1+C\Delta T_l) \| u(T_{n-1}) -U_{n-1}^{k+1}\| \\
     \| \mathcal{M}_{1,l}(u(T_{n-1}),\epsilon, \eta_l) - \mathcal{M}_{1,l}(U_{n-1}^k,\epsilon, \eta_l) \| &\le C \eta_l \epsilon \| u(T_{n-1}) -U_{n-1}^{k} \| \\
     \|\mathcal{M}_{0,l}(u(T_{n-1}), \epsilon, \eta_l, \Delta T)  -\mathcal{M}_{0,l}(U_{n-1}^k, \epsilon, \eta_l, \Delta T_l)\| &\le C \Delta T_l^{p+1} \max_{\omega_0 \le \omega} \left\{ \left|\frac{\epsilon}{\omega} \right| \kappa(\epsilon,\eta_l, \omega) \right\}
      \| u(T_{n-1}) -U_{n-1}^{k} \|. \label{eq:Lip_bound_M1}
\end{align}
The first equation is the stability of the propagator. A relation similar to \cref{eq:Lip_bound_M1} can be obtained using the argument applied to the first two lines of the second equation in the proof of \cref{thm:two-level_bound}.  For the estimate of the second equation see lemma \ref{lemma:Lemma_about_M1}.

Thus, the following estimate can be obtained for the two-level case 
\begin{align*}
    \|u(T_n) - U_n^{k+1}\| &\le \ C \eta_1 \epsilon \| u(T_{n-1}) -U_{n-1}^{k} \| \\
    & \quad + C \Delta T_1^{p+1} \max_{\omega_0 \le \omega} \left\{ \left|\frac{\epsilon}{\omega} \right| \kappa(\epsilon,\eta_1, \omega) \right\}
      \| u(T_{n-1}) -U_{n-1}^{k} \| \\
      & \quad + (1+C\Delta T_1) \| u(T_{n-1}) -U_{n-1}^{k+1}\| \\
      & \quad + c \Delta T_1 \Delta T_0^{p_0} \\
      & = \left(C \eta_1 \epsilon +  C \Delta T_1^{p+1} \max_{\omega_0 \le \omega} \left\{ \left|\frac{\epsilon}{\omega} \right| \kappa(\epsilon,\eta_1, \omega) \right\} \right) \| u(T_{n-1}) -U_{n-1}^{k} \| \\& \quad + (1+C\Delta T_1) \| u(T_{n-1}) -U_{n-1}^{k+1}\| \\
      & \quad + c \Delta T_1 \Delta T_0^{p_0} \\
      & = \bar \alpha_1 e_{n-1}^{k} + \beta_1 e_{n-1}^{k+1} +\delta_0
\end{align*}
which gives us a recurrence relation, where $k$ counts the Parareal iterations and $n$ counts the steps through the time grid.

For $k=0$, we have
\begin{align*}
    \|u(T_n) - U_n^{0}\| \le e_n^0 &= \|E(u(T_{n-1})) - \bar G^1(U_{n-1}^0)\| \\
    &\le \|E(u(T_{n-1})) - \bar G^1(u(T_{n-1}))\| +
    \|\bar G^1(u(T_{n-1})) - \bar G^1(U_{n-1}^0)\| \\
    &\le \|E(u(T_{n-1})) - \bar E^1(u(T_{n-1}))\| + \|\bar E^1(u(T_{n-1})) - \bar G^1(u(T_{n-1})) \| \  \\
    & \quad + \|\bar G^1(u(T_{n-1})) - \bar G^1(U_{n-1}^0)\| \\
    &\le C \eta_1 \epsilon \|\mathcal{\tilde M}_1\| + C \Delta T_1^{p_c+1} \max_{\omega_0 \le \omega} \left\{ \left|\frac{\epsilon}{\omega} \right| \kappa(\epsilon,\eta_1, \omega) \right\} \|\mathcal{\tilde M}_0\|  \\
    & \quad + (1+ C \Delta T_1) \|u(T_{n-1}) - U_{n-1}^{0}\| \\ 
    & = \bar \gamma_1 + \beta_1 e_{n-1}^0 \ .
\end{align*}


We have used the following notation further above
\begin{align*}
    \bar \alpha_l &= C \eta_l \epsilon + C \Delta T_l^{p_c+1} \max_{\omega_0 \le \omega} \left\{ \left|\frac{\epsilon}{\omega} \right| \kappa(\epsilon,\eta_l, \omega) \right\}\\
    \beta_l &= 1+C\Delta T_l\\
    \bar \gamma_l &= C \eta_l \epsilon \|\mathcal{\tilde M}_1\| + C \Delta T_l^{p_c+1} \max_{\omega_0 \le \omega} \left\{ \left|\frac{\epsilon}{\omega} \right| \kappa(\epsilon,\eta_l, \omega) \right\} \|\mathcal{\tilde M}_0\|\\
    \delta_0 &= c \Delta T_1 \Delta t^{p_0}
\end{align*}

We could find the same recurrence relation as in \cref{thm:two-level_bound} with mofified $\alpha$ and $\beta$. For a two-level result with non-exact fine propagator lemma \ref{lemma:reucrsion_equation_1} can be applied. The inductive step for the multi-level result can be found in lemma \ref{lemma:first_estimate_theorem2}.

{\color{white}1}
\end{proof}

\section{Complexity}
\label{sec:Efficency}

In this section, we distinguish between the total number of steps done on a level and the serial number of steps. Suppose on the coarsest level we do $X$ steps with the coarse propagator, for example to compute the initial guess. Then the total number of steps on that level is $X$ and the serial number of steps on that level is  $X$, too. Now, we go to the next finer level and the coarsening factor is $N$. We then have  $NX$ as the total number of steps and  $N$ as the number of serial steps (done on  $X$ grids in parallel). Refining again leads to $N^2X$ as the total number of steps and  $N$ serial steps (done on  $NX$ grids in parallel) and so on.

We define the complexity of an algorithm as the number of serial steps.
The complexity of the 3-level Parareal algorithm can be defined as
\begin{equation}
   C^3 = k_2 \left( N_2 + k_1 \left( N_1+N_0 \right) +N_1 \right) + N_2 ,
\end{equation}
where $N_i$ is the number of serial steps on level $i$ and $k_i$ is the number of iterations on level $i$. For a V-cycle we have $k_i=1$.
We may hypothesize that increasing the number of levels pays off when we have a broad range of scales or strong scale separation. 
 We investigate the computational complexity of a V-cycle in more detail.
Suppose $\tilde X$ is total number of fine steps on the finest level that must be done and $N$ is the coarsening factor which relates the time-steps of the different levels, i.e $\tilde X$ would be the number of serial steps that must be done when a non-parallelizable basic ODE solver is applied. (We might assume that two neighboring levels are related by the different coarsening factors. To keep the computation simple we will not do that in this place and leave it for future work.) We assume that we do $N$ serial steps on all level except the coarsest level. 
For the complexity of a V-cycle, we have
\begin{align}
    \text{1 level: } f_1(N) &= \tilde X \text{ serial steps} \\
    \text{2 levels: } f_2(N) &= N+2 \tilde  X/N \text{ serial steps} \\
    \text{3 levels: } f_3(N) &= N+2N+2 \tilde  X/N^2 = 3N+ 2 \tilde X/N^2 \text{ serial steps} \\
    \text{L levels: } f_L(N) &= 2(L-2)N+N+ 2 \tilde  X/N^{L-1} \text{ serial steps} . 
\end{align}

It is possible to find the coarsening factor $N$ which minimizes the number of serial steps, depending on the number of levels. Solving $f_L'(N) = 0$, we find

\begin{equation}
    N_{\mathrm{opt}} = (\tilde X + \tilde X/(2L-3))^{1/L} .
\end{equation}
However, when we choose a coarsening factor $N$ which is computed in the described way, we adapt the algorithm on the behavior of the model on the finest scale only. It may be necessary to account for the behavior of the model on other scales too, see the example in \cref{subsec:Swinging_spring}.

For a V-cycle, the total number of evaluations of the right-hand side for all the levels is not that much bigger than that of a serial time-stepper.
On the finest level we do the same number of   evaluations of the right-hand side that we would do with a serial time-stepping scheme, provided we apply the Multi-level Parareal algorithm with averaging. When we do in total $\tilde X$  evaluations of the right-hand side on the finest level, level 0, we do $2 \tilde X/N$  evaluations of the right-hand side  on level 1,  $2 \tilde X/N^2$  evaluations of the right-hand side on level 2 and so on. That means that the number of  evaluations of the right-hand side decreases exponentially with the levels.  Considering the  Multi-level Parareal algorithm without averaging we can possibly not coarsen in the same way because the time-step might become too large on the coarser levels and as a result the numerical solver might be unstable for the oscillatory problems.

\section{Numerical examples}
\label{sec:Numerical_examples}

For the numerical examples the explicit midpoint rule (RK2) is used as a basic integrator unless stated differently. The results presented in \cref{subsec:System_with_three_scales}, \cref{subsec:Swinging_spring} and \cref{subsec:RSWE}  are computed solving the modulation equations of the systems. Implementations of the examples can be found in \cite{juliane_rosemeier_2022_7382198}.

\subsection{1D single scale example}
\label{subsec:1D_singe_scale_example}

Let us consider the problem
\begin{equation}
\label{eq:non_oscillatory_problem}
    \frac{d x}{dt} = -x, \qquad x(0) = 1, \qquad t \in [0,2]
\end{equation}
which is to be solved with the Multi-level Parareal algorithm without averaging.
We compute the error for one V-cycle and increase the number of levels. The time-step on the coarsest level is for all the computations in  \cref{tab:error_coarse_dt_const_} constant $\Delta T = 0.25$. We choose a coarsening factor of $N=10$. This means: When we have a time-step $\Delta t$ on level $l$, the time-step on level $l-1$ is $\Delta t/10$.
        
\begin{table}[]
\centering
\begin{tabular}{ c| c }
number of levels & error  \\ \hline
2 & 1.2566212807763046e-05 \\
3 & 1.9562958164422008e-05 \\
4 &1.9807099440426344e-05\\
5 & 1.9809587023590493e-05\\
6 & 1.9809615854133382e-05\\
7 & 1.9809616125891306e-05\\
8 & 1.980961620086837e-05\\
\end{tabular}
\caption{Error of the Multi-level Parareal algorithm without averaging for a varying number of levels applied to problem \eqref{eq:non_oscillatory_problem} where the step size on the coarsest level is 0.25 . We see the relative error in the discrete l1 norm.}
\label{tab:error_coarse_dt_const_}
\end{table}

We see in \cref{tab:error_coarse_dt_const_} that increasing the number of levels while keeping the time-step on the coarsest level constant barely changes the error of the approximations of problem \cref{eq:non_oscillatory_problem}.

This example leads to the following observation. The Multi-level Parareal algorithm with $L$ levels (without averaging) is a two-level Parareal algorithm where the fine propagator is a Parareal algorithm with $L-1$ levels. Therefore, when the dominant component of the error emerges on the coarsest level, we cannot not expect that the accuracy of the method increases when we only add additional levels and keep the time-step on the coarsest level constant.

\subsection{1D example with fast oscillations}
\label{sucsec:1D_example_with_fast_oscillations}

Let us now suppose the following situation. We have to solve a multi-scale problem with very fast oscillations on the finest scale. On the finest level we resolve the finest scale. Therefore the time-step must be chosen small enough to resolve the very fast oscillations. This is independent of the total number of  levels that the method has. In particular, the fastest scale dictates the time-step on level 0. 

An example for such a problem is given by
\begin{equation}
\label{eq:test_problem_eff}
    \frac{d w}{dt} = - \exp(irt)w^2
\end{equation}
for large values of $r$. This equation can be interpreted as a one dimensional modulation equation. In the following we will consider $r=100$, $r=1000$ and $r=10000$. The exact solution of the problem is given by
\begin{equation}
    w(t) = \frac{r w_0}{-iw_0 \exp(irt) + iw_0 +r} .
\end{equation}
The averaged problem obeys the following relation
\begin{equation}
    \frac{d \bar w}{dt} = (-1)\ \underbrace{\frac{1}{\eta_l} \int_{-\eta_l/2}^{\eta_l/2} \ \rho \left( \frac{s}{\eta_l} \right) \
\exp(irs) \ ds}_{\text{damping factor}} \ \exp(irt) \ \bar w(t)^2 .
\end{equation}

The problem \eqref{eq:test_problem_eff} will be solved with Multi-level Parareal methods including averaging.
Numerical studies show that for $r = 100$ the fastest period in the solution is of order $0.1$. Therefore we choose averaging windows 
$\eta_1 = 0.2 \text{ and } \eta_2 = 2$. In addition, for $r = 1000$ the fastest period in the solution is of order $0.01$. Thus the averaging windows 
$\eta_1 = 0.02, \eta_2 = 0.2 \text{ and } \eta_3 = 2 $ are a reasonable choice.
For the case $r = 10000$ the fastest period in the solution is of order $0.001$. Therefore, the computations are done with the averaging windows 
$\eta_1 = 0.002, \eta_2 = 0.02, \eta_3 = 0.2 \text{ and } \eta_4 = 2 $.  The length of the solution interval is always 1 and the coarsening factor is always 10. To solve problem \cref{eq:test_problem_eff} a V-cycle is applied. Depending on the value for $r$ we choose different step size on the finest level 
\begin{itemize}
    \item r = 100:  The step size on the finest level is $\Delta T_0 = 10^{-3}$
    \item r = 1000: The  step size on the finest level is $\Delta T_0 = 10^{-4}$
    \item r = 10000: The  step size on the finest level is $\Delta T_0 = 2.5 \cdot 10^{-5}$
\end{itemize}

\begin{table}[]
\centering
\begin{tabular}{ c | c| c| c }
 number of levels & $r=100$ & $r=1000$ &  $r=10000$ \\ \hline
 2 & 0.0002169750591733674 & 2.1847140061040485e-06 &  3.0668862104273734e-07 \\
 3 & 0.00019811199764541986 & 2.106413305540747e-06 & 3.0480547757705495e-07  \\
 4 & & 2.251750942815333e-06 & 3.0357016877934065e-07\\
 5 & & & 2.7011063136189545e-07\\
\end{tabular}
\caption{Errors at time $t=1$ for a V-cycle including averaging with a varying number of levels applied to problem \eqref{eq:test_problem_eff} with $r=100; 1000 \text{ and }10000$. The time-step on the finest level is $\Delta t = 1 \cdot 10^{3}, 1 \cdot 10^{-4} \text{ and } 2.5 \cdot 10^{-5}$. The coarsening factor to relate the levels is $N=10$.}
\label{tab:error_coarse_dt_const_5}
\end{table}

In  
\cref{tab:error_coarse_dt_const_5} we find the errors at time $t=1$. The time-step on the finest level was adapted to the fast oscillations in the problem which depend on the parameter $r$. Fixing one value for $r$, we see that the errors  hardly change in the number of levels. We can draw the following conclusions. First, for the parameters chosen for the model runs the increased parallelism of the methods with 3, 4 or 5 levels should lead to a better efficiency. Second, having a third, fourth or fifth level does not change the error significantly. This might indicate that the error emerges on the finest level, where the fast oscillations are resolved and computed numerically.

This means when we want to design more efficient algorithms with multi-level methods, we have to exploit the increased parallelism when we have multiple levels. A gain in efficiency should come from a combination of increased parallelism and larger time-steps on the coarser levels. For oscillatory problems larger times steps can lead to unstable behavior of the algorithm. This can be circumvented when the averaging is applied. Doing several Parareal iterations has an effect on the efficiency, too. 

\subsection{System with three scales}
\label{subsec:System_with_three_scales}

 We consider a simple model problem with a quadratic non-linearity which is inspired by an equation used to explain issues  due to bi-linear terms arising
    in fluid modelling, see \cite{Vallis_2017},
      \begin{equation}
        \begin{pmatrix} \dot u_1 \\ \dot u_2
          \\ \dot u_3 \end{pmatrix} +
        \begin{pmatrix} i \omega_1 &0& 0 \\0 & i \omega_2 &  0 \\0& 0 &
          i \omega_3 \end{pmatrix}
        \begin{pmatrix} u_1\\ u_2 \\ u_3 \end{pmatrix} +
        \begin{pmatrix} u_1u_1\\ u_2u_2 \\
            u_3u_3 \end{pmatrix} = 0 .
            \label{eq:three_timescales}
      \end{equation}
      We write $\dot u_1$ instead of $\frac{du_1}{dt}$.
     For strongly differing $\omega_i$ , $i \in \{1,2,3\}$, the system has three well-separated time scales. Example \eqref{eq:three_timescales} is very convenient for numerical studies because the exact solution of the system  and the averaged equations can be derived. 
     
     The system is decoupled which would in principle allow us to treat the equations separately. However, we want to study the effect of scale separation in a problem and in general multi-scale systems it can be the case that one equation sees the effect of several scales, see for instance the example in \cref{subsec:Swinging_spring}. 
     Therefore, when solving the system, we neglect the decoupling. Especially, we choose the same time-steps and averaging windows for all the equations.

\begin{figure}
     \centering
         \centering
         \includegraphics[width=1.\textwidth]{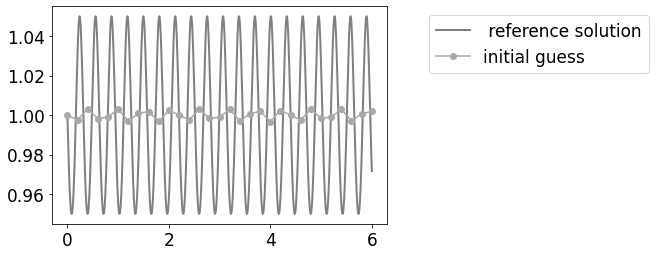}
         \caption{Illustration of the second component. The initial guess has strongly damped oscillations. Therefore the fine propagator gets almost the same values for each interval.}
         \label{fig:secondComp_InitGuess}
\end{figure}

\begin{figure}
         \centering
         \includegraphics[width=1.\textwidth]{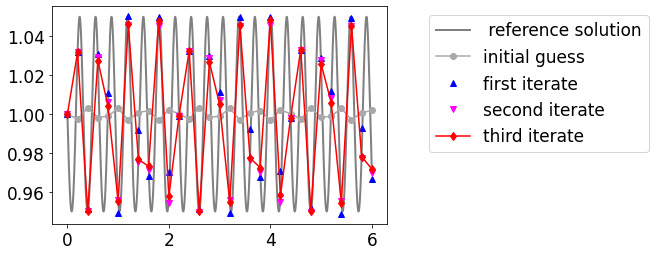}
         \caption{Illustration of the second component. The solutions obtained with the Parareal iterations reproduce the phase correctly, because the information about the phase is not in the initial condition but in the exponential, explicitely in the time coordinate $t$.}
         \label{fig:secondComp_ThirdIt}
\end{figure}

In \cref{sec:introduction}, we have already mentioned  that when working with the modulation equation, the information about the phase is not in the initial data but in the exponential which is part of the right-hand side of the system. This is now illustrated in \cref{fig:secondComp_InitGuess} and \cref{fig:secondComp_ThirdIt}. The figures show solutions of the second component of the system. In \cref{fig:secondComp_InitGuess} we see the strongly damped initial guess provided by the coarse propagator. This means that the fine propagator computes fine solutions with almost the same initial value on each small interval. Moreover, we see that due to the oscillations the exact solution, denoted as the reference solution in the figure, can deviate strongly from the damped initial guess. As \cref{fig:secondComp_ThirdIt} shows, the Parareal method with averaging computes the phase correctly.  But as the coarse propagator only provides strongly damped solutions, it must be the fine propagator which gives the strong variation from the damped solution in the Parareal iteration.
Therefore, as the fine propagator uses almost the same initial value on each small interval in the first iteration, the information about the phase must be in the right-hand side of the system and cannot be in the initial data.

\begin{figure}
    \centering
    \includegraphics[width=1.\textwidth]{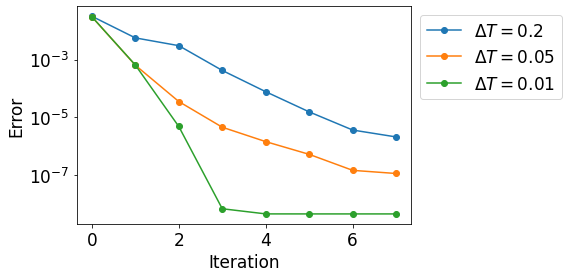}
    \caption{Errors of a three-level Parareal algorithm with averging for several iterations on the coarsest level, level 2. The time step on the coarsest level is denoted as $\Delta T$. We see the eerors of the second component of the system \eqref{eq:three_timescales}. }
    \label{fig:eq:three_timescales_error}
\end{figure}


\cref{fig:eq:three_timescales_error} shows a study of the numerical errors made with a 3-level Parareal algorithm with averaging. The parameters in problem \eqref{eq:three_timescales} are $\omega_1 =2, \omega_2 =20 \text{ and } \omega_3 =200$ and the solution is computed on the time interval [0,6]. The following  
averaging windows are chosen $\eta_1 = 0.1 \text{ and } \eta_2 = 1$.  In addition, a coarsening factor $N=10$ relates the different levels. We always have $k_1=3$ iteration on level 1. In \cref{fig:eq:three_timescales_error}, the error is computed at time $t=6$ for the second component.  We can see that the accuracy of the solution increases with the number of iterations on level 2 for different coarse time-steps on level 2. After several iterations the accuracy of the method does not increase anymore. The Multi-level Parareal method converges to the fine solution. In the case of the 3-level method with averaging the fine solution is the 2-level method with averaging.


\subsection{Swinging spring (Elastic Pendulum)}
\label{subsec:Swinging_spring}
 In this subsection we do a first test on a  fluid-related, more complicated problem, the swinging spring, also denoted as the elastic pendulum. The relation of the system to geophysical flows was outlined for instance in \cite{Lynch_2003}. A perturbation analysis for  the barotropic potential vorticity equation, a comlplex partial differential equation, leads to the three-wave equations. These equations also govern the dynamics of the swinging spring. Further investigations on the dynamics of the system can be found for example in \cite{Holm_Lynch_2002}.
 The swinging spring can be rewritten as a first order system and is given by
\begin{equation}
 \frac{d \bf w}{dt} =
\begin{pmatrix} \dot x_1 \\ \dot x_2 \\\dot y_1 \\ \dot y_2\\ \dot z_1 \\ \dot z_2 \end{pmatrix} =
\begin{pmatrix} 0 & 1& 0 & 0 & 0& 0 \\
- \omega_R^2 & 0& 0 & 0 & 0& 0 \\
0 & 0& 0 & 1 & 0& 0 \\
0 & 0& - \omega_R^2 & 0&0 &  0 \\
 0 & 0& 0 & 0&0 &  1 \\
0 & 0& 0 & 0 & - \omega_Z^2& 0 \\
\end{pmatrix} 
\begin{pmatrix} x_1 \\  x_2 \\ y_1 \\  y_2\\ z_1 \\ z_2 \end{pmatrix} +
\begin{pmatrix} 0 \\  \lambda x_1 z_1 \\ 0 \\  \lambda y_1 z_1\\ 0 \\ \frac{1}{2} \lambda (x_1^2 + y_1^2) \end{pmatrix}
 .
\end{equation}

Multi-level Parareal algorithms with averaging are applied to solve the system.
We compare a three-level method with a two-level method and investigate the errors at time $T_{max} =50$ for the first component of the system. For the two-level method we choose time-steps $\Delta T_1 = 5$ and $\Delta T_0 = 0.05$. The time-step $\Delta T_0$ is small enough to resolve the fast oscillations in the problem. With the time-step $\Delta T_1$ the coarse dynamics is resolved. Additionally, the averaging window is $\eta = 2$, which guarantees that the fast oscillations are averaged on the coarse level, level 1. 

When applying the three-level method, two iterations are done on the intermediate level, level 1. The number of iterations on the coarse level, level 2 is varied in the test. Moreover, the coarse and fine time-steps are the same as in the two-level case. Especially, we have $\Delta T_2 = 5$ and $\Delta T_0 = 0.05$. But we introduce an intermediate level and increase the parallelism. The time-step of the intermediate level, level 1, is $\Delta T_1 = 0.5$. In particular, the coarsening factor is $N=10$. For the coarsest level the averaging window is chosen as $\eta_2 = 2$. We test different averaging windows on the intermediate level, namely $\eta_1 = 0.2, 0.75 \text{ and } 2$. When we choose $\eta_1 = 0.2$ we relate the averaging windows on the levels by the same coarsening factor that relates the time-steps.  The choice $\eta_1 = 2$ is inspired by the dynamics of the system, because this window ensures that we average the fast oscillations also on level 1. The value $\eta_1 = 0.75$ is an intermediate choice to see how the error develops for changing averaging windows.

The reference solution is computed with the RK2 method with a time-step of $\Delta t = 0.001$. This is more accurate than the fine solver of the two-level method. The errors for the two- and three-level methods can be found in \cref{tab:error_coarse_dt_const_6}. We may observe that with an increasing number iterations on the coarse level the accuracy of both the two- and three-level methods increases until the methods are converged. How good the approximations of the three-level methods are strongly depends on the choice of the averaging window on the intermediate level, $\eta_1$. For $\eta_1 =2$ the convergence for the three-level method is almost the same as for the two-level method, whereas the accuracy becomes worse for smaller intermediate averaging windows. Consequently for the case $\eta_1 =2$ we can benefit from more parallelism of the three-level method. \cref{tab:error_coarse_dt_const_7} shows the number of serial steps done for the two- and three-level methods depending on the iterations. Especially, we see that we do less serial steps when the three-level method is applied.


\begin{table}[]
\centering
\begin{tabular}{ c | c | c| c| c }
 iteration & $\eta_1 = 0.2$ & $\eta_1 = 0.75$ & $\eta_1 = 2$ & 2-level method  \\ \hline
 1 & 0.000596698264213838 & 0.000708281803419343 & 0.0007175950573062584 & 0.0007175766438230063 \\
 2 & 0.00014503391511486857 & 4.627682726392884e-06 & 7.244961778282016e-06 & 7.2234487876014775e-06 \\
 3 & 0.00016846111177755765& 3.6002717718981725e-05 & 2.4794370887603473e-05 & 2.4824342166950703e-05 \\
 4 & 0.00015175374381607917& 1.977436923063236e-05 & 8.614816726031094e-06 & 8.641258080730602e-06 \\
 5 & 0.00015190431144367772& 1.980929225765815e-05 & 8.639280904016583e-06 & 8.665761458456767e-06 \\
\end{tabular}
\caption{Errors of the first component at time $T_{\mathrm{max}}=50$.}
\label{tab:error_coarse_dt_const_6}
\end{table}

\begin{table}[]
\centering
\begin{tabular}{ c | c | c }
 iteration & 2-level method & 3-level method \\  \hline
 1 & 120 & 70 \\
 2 & 230 & 130 \\
 3 & 340 & 190 \\
 4 & 450 & 250 \\
 5 & 560 & 310 \\
\end{tabular}
\caption{Computational complexity. Number of serial steps for a two- and three-level method. For the three-level method the number of iterations on level 1 is $k_1=2$. Varying number of iterations on the coarsest level for both the two- and three-level method.}
\label{tab:error_coarse_dt_const_7}
\end{table}


  


  



\subsection{1D Rotating Shallow Water Equations (RSWE)}
\label{subsec:RSWE}
 
  Here solutions to the RSWE given by the equations
  \begin{align}
      \frac{dv_1}{dt} + \frac{1}{\epsilon} \left( -v_2+ F^{-1/2}
      \frac{dh}{dx}\right) + v_1 \frac{dv_1}{dx} &=\mu \partial_x^4 v_1 \\
       \frac{dv_2}{dt} + \frac{1}{\epsilon} v_1 + v_1 \frac{dv_2}{dx} &= \mu \partial_x^4 v_2
      \\
       \frac{dh}{dt} + \frac{ F^{-1/2}}{\epsilon} \frac{dv_1}{dx} +
      \frac{\partial}{\partial x} (hv_1) &= \mu \partial_x^4 h
  \end{align}
 are computed. This is the same example investigated in \cite{Haut_Wingate_2014} where convergence of the two-level method with averaging is shown. A similar notation is used here. With $h(x,t)$ we denote the surface height and $v1(x,t), v2(x,t)$ are the horizontal velocity.  
 Moreover, a hyperviscosity term with diffusion coefficient $\mu = 10^{-4}$ is used.   Additionally, we impose periodic boundary conditions and use the following initial data
 \begin{equation}
 \begin{split}
 & v1=0, \qquad v2=0 \\
 & h = c_1 \left( e^{(-4(x-\pi/4)^2)} \sin(3(x-\pi/2)) + e^{(-2(x-\pi)^2)} \sin(8(x-\pi)) \right) + c_0, 
 \end{split}
  \end{equation}
 where the constants $c_1 \text{ and } c_2$ are chosen such such that
 \begin{equation}
     \int_0^{2\pi} h(x,0) dx = 0, \qquad \max_x | h(x,0) | = 1 .
 \end{equation}
 
 To solve the RSWE numerically a  pseudo-spectral method with 128 spatial Fourier modes is be applied. For the time-stepping we use a three-level Parareal scheme with averaging and with a second order Strang splitting method as basic integrators. The numerical time-stepping method is tested for two different parameter regimes. The Rossby number $\epsilon$ is chosen 0.1 in both cases. The Froude number is given by the relation $F^{1/2} \epsilon$. In the first test   the Burger number $F$ is chosen to be 1. This test is done with coarsening factors $N=10,20,40$ and an interval length 48. In the second test the Burger number $F$ is chosen to be 1/100. The coarsening factors are $N=10,20,30$ and the interval length is 45. For both tests, the time-step on the finest level is 1/2000. The time-steps on level 1 and 2 are determined by the fine time-step and the coarsening factor. The averaging windows on the levels 1 and 2 are chosen to be equal to the time-steps on the levels 1 and 2.  In \cref{fig:3level_error_quasigeostophic} and \cref{fig:3level_error_incompressible} the  errors in the relative $L^{\infty}$ norm, the same norm used for the errors in \cite{Haut_Wingate_2014}, are shown. We  can see that the numerical method converges with  increasing number iterations. For the case $F=1/100$ fewer iterations are needed for convergence compared to the case $F=1$ .

 \begin{figure}
     \centering
     \includegraphics[scale=0.8]{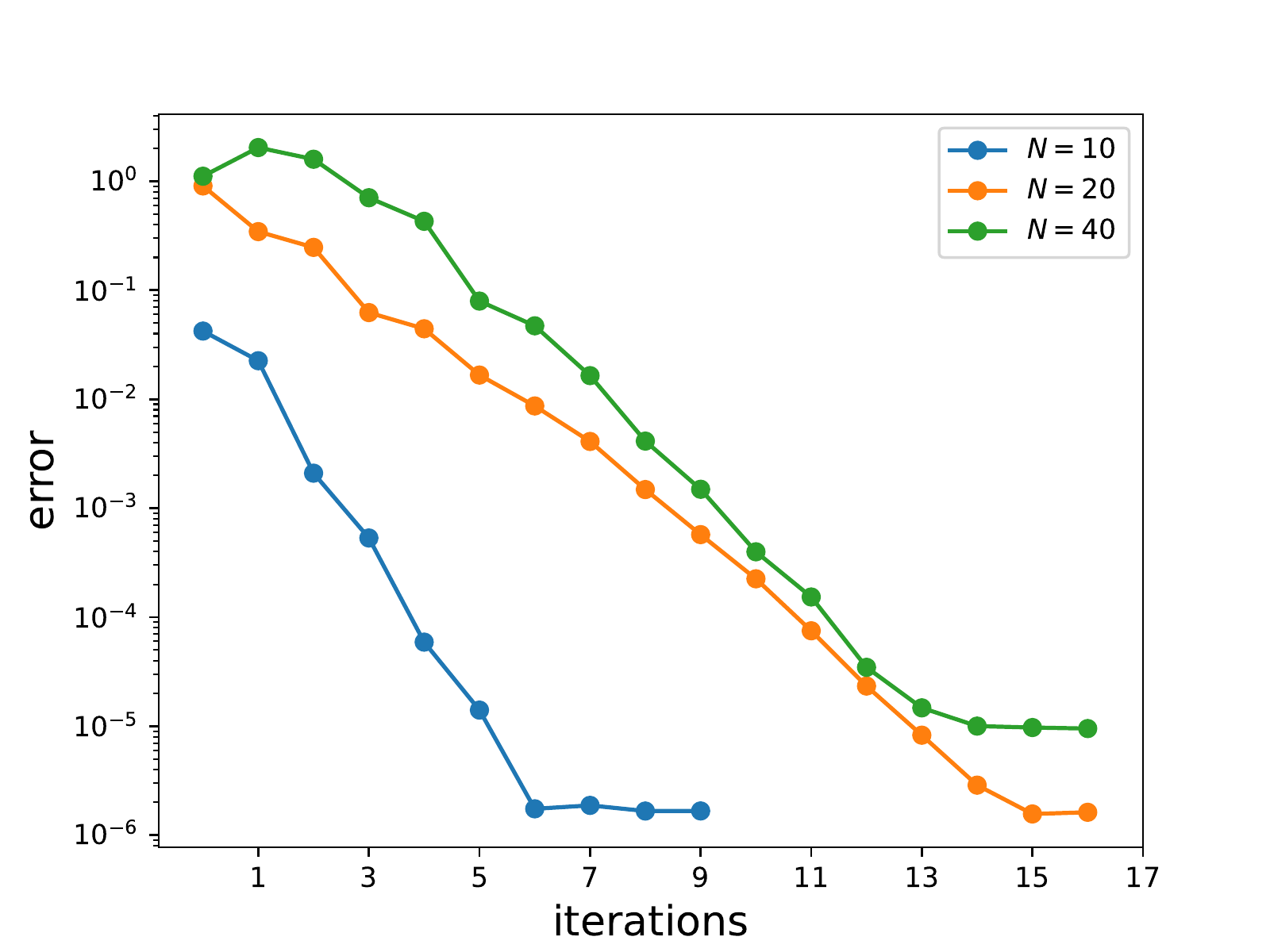}
     \caption{Error of the numerical solution to the RSWE with a three-level Parareal method with averaging, parameter choice: $\epsilon = 0.1, F=1$}
     \label{fig:3level_error_quasigeostophic}
 \end{figure}
 
 \begin{figure}
     \centering
     \includegraphics[scale=0.8]{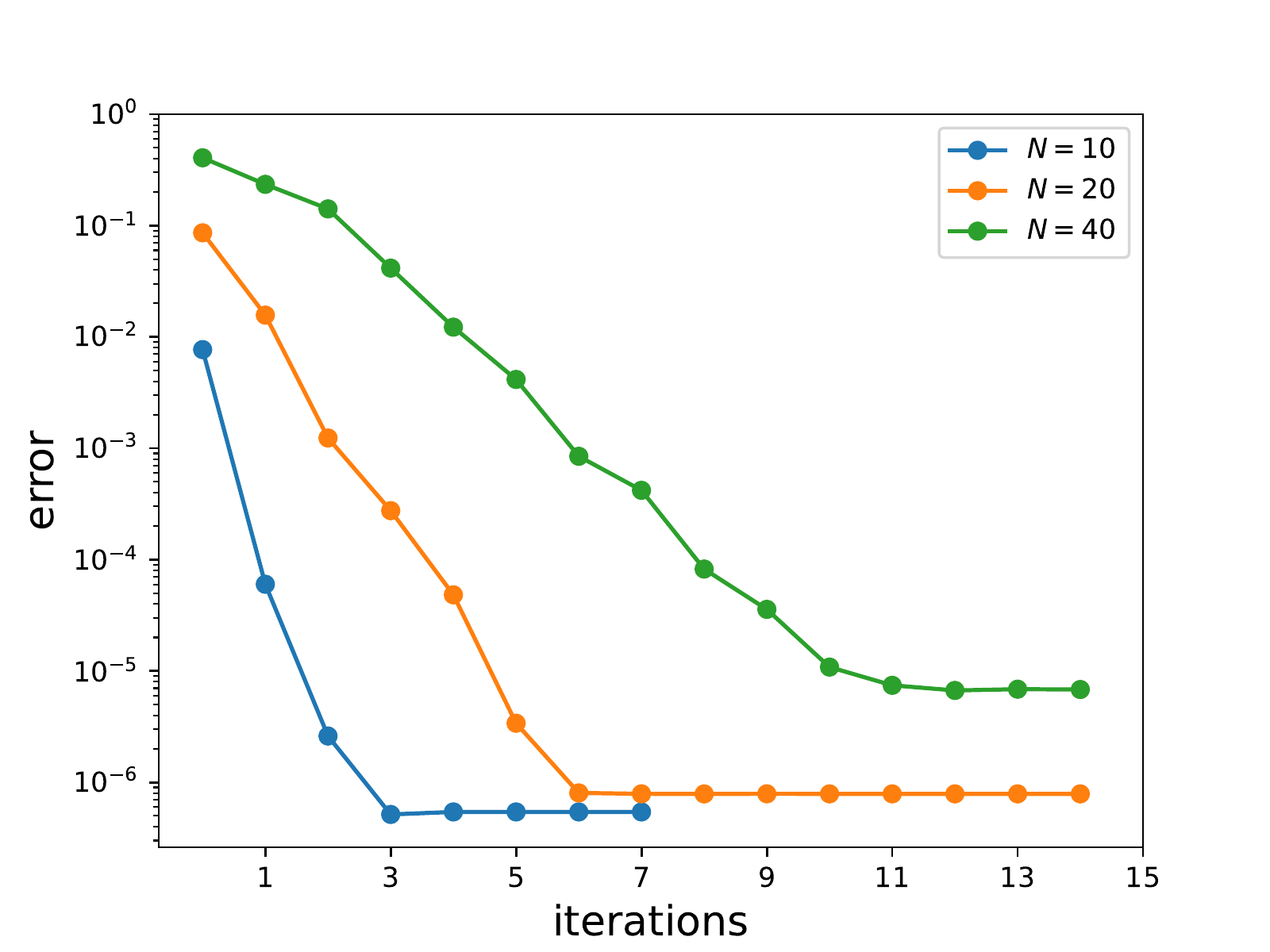}
     \caption{Error of the numerical solution to the RSWE with a three-level Parareal method with averaging, parameter choice: $\epsilon = 0.1, F=1/100$}
     \label{fig:3level_error_incompressible}
 \end{figure}
 
Suppose we are given a two-level method to solve the case with $F=1$. The two-level method uses a coarsening factor $N=40$ and two correction iterations are done. The solution interval in time is [0,48]. Then we want to answer the question if it is possible to design a three-level method which is more efficient than the two-level method. The two-level method gives an error of 1.7963565539455182e-05, where 7280 serial steps are done. We compare it to the three-level method with a coarsening factor of $N=20$. With $k_2$ we denote the number of iterations on the coarse level and we do 3 iterations on the intermediate level. Then the three-level method converges to a two-level method with three iterations, which guarantees that the three-level method converges to a two-level method, which is more exact than the  two-level method with two iterations. When $k<28$, we do less serial steps with the three-level method than with two-level method with two iterations. With the three-level method we reach the accuracy of the two-level method with two iterations after 13 iterations already. Therefore, we can find a three-level method which is more efficient than two-level method with two iterations.

 Now we consider the case $N=1/100$. We choose a two-level method with a coarsening factor of $N=60$ and do two correction iterations. Thus, 4560 serial steps must be done. We compare this method to the three-level method with a coarsening factor of $N=30$ and $k_1=3$ iterations on level 1. Thus, depending on the number of iterations  $k_2$ we do $310k_2+100$ serial steps with three-level method. If we do not more that 13 iterations with the three-level method, the computational complexity of the three-level method is less than the computational complexity of the two-level method. The three-level method is more accurate than the two-level method after 8 iterations already, as the two-level method gives an error of 0.00019024269235007262.

\section{Discussion and conclusion}
\label{sec:Discussion_conclusion}

In the present study, a  Multi-level Parareal method with and without  averaging   is proposed and investigated, with a special focus on oscillatory problems.  
The averaging plays a central role when oscillatory problems shall be solved because it mitigates the oscillatory stiffness. Therefore big time-steps can be used on the coarse levels, which is  demonstrated in the numerical examples. For future studies  a change of the integration kernels might be of interest to adapt the method to dissipative problems. 

One important result is the derivation of  error bounds for the method building on results in the literature and underpinning theoretically the convergence of the method. In particular, the basic time-stepping methods applied on the levels cause an error, denoted as $E_l$ or $\delta_0$. The errors are amplified as a solution from one level is passed to coarser grids. In the error bound the amplification is given by the amplification factors $A_j$. Increasing the number of iterations on a level $l$ changes  $E_l$. Additionally, the computational complexity of the Multi-level Parareal method is discussed in one of the sections. 

Finally, several numerical examples are studied. The investigation of the numerical examples includes a discussion when multiple levels can be more efficient than the Parareal method with only two levels. A gain of efficiency can be difficult to accomplish if the algorithms without averaging are used, see the example in \cref{subsec:1D_singe_scale_example}. However, combined with averaging the multi-level approach seems promising for oscillatory problems, as it allows us to take big time-steps on the coarse levels and we do not have to satisfy severe stability constraints imposed by fast oscillations, see the other examples. The example in \cref{subsec:Swinging_spring} also shows that a good choice of the averaging windows is crucial for the accuracy. Moreover, the computations in \cref{subsec:RSWE} show that multi-level methods can be more efficient than a given two-level method.

The examinations on the example in \cref{subsec:Swinging_spring} also inspire an interesting variant of the Multi-level Parareal method. When introducing an intermediate level, we might use a constant function as initial guess on the grids of level 1. In particular, the constant functions on level 1 would admit the initial values provided by level 2. Possibly, this would not increase the error significantly, because the level 1 solutions only have to capture the coarse dynamics which is already resolved by the solution from level 2. Additionally, we would save the serial steps needed to compute the initial guess on level 1. The investigations concerning computational complexity and efficiency are not exhaustive and shall be continued in the future on more complex examples including many scales or a continuous range of scales.

\appendix

\section{Details of the convergence proofs}

\subsection{Non-exact fine propagator}

\begin{lemma}
\label{lemma:reucrsion_equation_1}
Suppose we are given the following recurrence relation
\begin{equation*}
    e_n^{k+1} = \alpha e_{n-1}^k + \beta e_{n-1}^{k+1} + \delta_n, \qquad \qquad
    e_n^0 = \gamma + \beta e_{n-1}^0 ,
\end{equation*}
with $\delta_0 = 0$ and $\delta_n = \delta$ for $n \ge 1$.
Then the $e_n^{k}$ can be written in non-recursive form as
\begin{equation}
     e_n^k \le {n \choose k+1}    \gamma \alpha^k  \beta^{n-k-1} + n \delta \beta^{n-1} (1+\alpha_{0n})^{n-1}
\end{equation}
for $n>k$.
\end{lemma}

\begin{proof}
The recurrence relation has two indices $k$ and $n$. To eliminate the index $n$ we introduce the generating function $\rho_k$ defined as
\begin{equation*}
    \rho_k = \sum_{n \ge 1} e_n^k \zeta^n,
\end{equation*}
see for example \cite{Wilf_1990}. Now a recurrence relation for the $\rho_k$ can be derived. We start with $k=0$.
\begin{align*}
    &\rho_0 = \beta \sum_{n \ge 1} e_{n-1}^0 \zeta^n + \gamma \sum_{n\ge 1} \zeta^n
    = \beta \zeta \rho_0 + \gamma \frac{\zeta}{1- \zeta} \\
   &\rho_0 = \frac{1}{1-\beta \zeta} \frac{\gamma \zeta}{1-\zeta},
\end{align*}
For $k \ge 1$ we find the following identity.
\begin{align*}
    \rho_{k+1} & = \alpha \sum_{n \ge 1} e_{n-1}^k \zeta^n +
    \beta \sum_{n \ge 1} e_{n-1}^{k+1} \zeta^n
    + \delta \frac{\zeta}{1-\zeta} \\
    & = \alpha \zeta \rho_k + \beta \zeta \rho_{k+1} + \delta \frac{\zeta}{1-\zeta} 
\end{align*}
or
\begin{equation*}
    \rho_{k+1} = \frac{\alpha \zeta}{1- \beta \zeta} \rho_k + \frac{\delta}{1- \beta \zeta} \frac{\zeta}{1-\zeta}
\end{equation*}
 This is a recursion for $\rho_k$. We introduce the following notation:
 $d_1 = \frac{\alpha \zeta}{1- \beta \zeta}$ and $d_2 = \frac{\delta}{1- \beta \zeta} \frac{\zeta}{1-\zeta}$. The recursion for the $\rho_k$ can therefore be rewritten as
 \begin{equation*}
    \rho_{k+1} = d_1 \rho_{k} + d_2 .  
 \end{equation*}
A non-recursive form is given by the the following relation, which can be shown using induction 
\begin{align*}
    \rho_{k} & = d_1^k \rho_0 + d_2 \sum_{r=0}^{k-1} d_1^r \\
    & = \left(\frac{\alpha \zeta}{1- \beta \zeta}\right)^k \rho_0 + \frac{\delta}{1- \beta \zeta} \frac{\zeta}{1-\zeta}
    \sum_{r=0}^{k-1} \left(\frac{\alpha \zeta}{1- \beta \zeta} \right)^r \\
    & = \left(\frac{\alpha \zeta}{1- \beta \zeta}\right)^k \frac{1}{1-\beta \zeta} \frac{\gamma \zeta}{1-\zeta} 
    + \frac{\delta}{1- \beta \zeta} \frac{\zeta}{1-\zeta}
    \sum_{r=0}^{k-1} \left(\frac{\alpha \zeta}{1- \beta \zeta} \right)^r \\
    & = \gamma \alpha^k \frac{\zeta^{k+1}}{1-\zeta} \frac{1}{(1-\beta\zeta)^{k+1}}
    +  \frac{\delta \zeta}{1-\zeta}
    \sum_{r=0}^{k-1}(\frac{(\alpha \zeta)^r}{(1- \beta \zeta)^{r+1}} 
\end{align*}
A bound for the first term in the sum is given in lemma \ref{lemma:estimate_with_binomial_theorem1}. The second term is treated separately in lemma \ref{lemma:estimate_with_binomial_theorem2}. Altogether we can bound $\rho_k$ by

\begin{equation*}
   \rho_{k} \le \sum_{j \ge k+1} {j \choose k+1} \gamma \alpha^k  \beta^{j-k-1} \zeta^{j}  + \sum_{n=1}^{\infty} n \delta \beta^{n-1} (1+\alpha_{0n})^{n-1} \zeta^n .
\end{equation*}

The generating function $\rho_k$ was defined as $\rho_k = \sum_{n \ge 1} e_n^k \zeta^n$ in the beginning. Using the last relation and equating the coefficients, for $n>k$ the error can be estimated by
\begin{align*}
    e_n^k \le {n \choose k+1}    \gamma \alpha^k  \beta^{n-k-1} + n \delta \beta^{n-1} (1+\alpha_{0n})^{n-1} . 
\end{align*}

Remark: For $n\le k$ only the fine propagator contributes to the error.
\end{proof}

\begin{lemma}
\label{lemma:estimate_with_binomial_theorem1}
Suppose $\beta \ge 1$ and $0<\zeta< 1/\beta$. Then the following inequality holds
\begin{equation}
\label{eq:estimate_with_binomial_theorem1}
    \gamma \alpha^k \frac{\zeta^{k+1}}{1-\zeta} \frac{1}{(1-\beta\zeta)^{k+1}} \le
    \sum_{j \ge k+1} {j \choose k+1}
    \gamma \alpha^k  \beta^{j-k-1} \zeta^{j} .
\end{equation}
\end{lemma}

\begin{proof}
Assuming that $0< \zeta < 1/\beta$, we have 
\begin{equation*}
\gamma \alpha^k \frac{\zeta^{k+1}}{1-\zeta} \frac{1}{(1-\beta\zeta)^{k+1}} \le
    \gamma \alpha^k  \frac{\zeta^{k+1}}{(1-\beta\zeta)^{k+2}} .
\end{equation*}
Applying the binomial theorem yields
\begin{equation*}
    \frac{1}{(1-\beta\zeta)^{k+2}} =
    \sum_{j \ge 0} {k+1+j \choose j} \beta^j
    \zeta^j .
\end{equation*}
Therefore, we can derive the following relation
\begin{align*}
    \gamma \alpha^k  \frac{\zeta^{k+1}}{(1-\beta\zeta)^{k+2}}
    & = \sum_{j \ge 0} {k+1+j \choose j} \gamma \alpha^k  \beta^j \zeta^{k+1+j} \\
    & = \sum_{j \ge k+1} {j \choose j -(k+1)}
    \gamma \alpha^k  \beta^{j-(k+1)} \zeta^{j} \\
    & = \sum_{j \ge k+1} {j \choose k+1}
    \gamma \alpha^k  \beta^{j-k-1} \zeta^{j} .
\end{align*}

\end{proof}

\begin{lemma}
\label{lemma:estimate_with_binomial_theorem2}
Suppose we are given $\beta > 1$. Choose $\zeta$ such that $\beta \zeta <1$ and $\frac{\alpha \zeta}{1-\beta\zeta} < 1$ Then, we can find the following estimate
\begin{equation*}
    \frac{\delta \zeta}{1-\zeta}
    \sum_{r=0}^{k-1}(\frac{(\alpha \zeta)^r}{(1- \beta \zeta)^{r+1}}  \le
    \sum_{n=1}^{\infty} n \delta \beta^{n-1} (1+\alpha_{0n})^{n-1} \zeta^n ,
\end{equation*}
where $\alpha_{0n} \in [0, \alpha]$.
\end{lemma}

\begin{proof}
We have 
\begin{align*}
    \frac{\delta \zeta}{1-\zeta}
    \sum_{r=0}^{k-1}(\frac{(\alpha \zeta)^r}{(1- \beta \zeta)^{r+1}} & \le
    \sum_{r=0}^{k-1}\frac{\delta \alpha^r \zeta^{r+1}}{(1- \beta \zeta)^{r+2}} \\
    & = \sum_{r=0}^{k-1}\delta \alpha^r \zeta^{r+1} \sum_{j \ge 0} {r+1+j \choose j} \beta^j \zeta^{j} \\
    & = \sum_{r=0}^{k-1} \sum_{j \ge 0}
    {r+1+j \choose j} \delta \alpha^r \beta^j \zeta^{j+r+1} \\
    & = \sum_{r=0}^{k-1} \sum_{j \ge r+1}
    {j \choose r+1} \delta \alpha^r \beta^{j-r-1} \zeta^{j} 
\end{align*}
In the second line Newton's generalized binomial theorem is applied.

For the summands we define
\begin{equation*}
    S_r = \alpha^r \delta \sum_{j \ge r+1}
    {j \choose r+1}  \beta^{j-r-1} \zeta^{j} .
\end{equation*}
In particular, we have
\begin{align*}
    S_0 & = \delta \sum_{j \ge 1} {j \choose 1} \beta ^{j-1} \zeta^j \\
    & = \delta \left( {1 \choose 1} \zeta + 
    {2 \choose 1} \beta \zeta^2 +
     {3 \choose 1} \beta^2 \zeta^3 + \dots \right)
\end{align*}
\begin{align*}
    S_1 & = \alpha \delta \sum_{j \ge 2} {j \choose 2} \beta ^{j-2} \zeta^j \\
    & = \alpha \delta \left( {2 \choose 2} \zeta^2 + 
    {3 \choose 2} \beta \zeta^3 +
     {4 \choose 2} \beta^2 \zeta^4 + \dots \right)
\end{align*}
\begin{align*}
    S_2 & = \alpha^2 \delta \sum_{j \ge 3} {j \choose 3} \beta ^{j-3} \zeta^j \\
    & = \alpha^2 \delta \left( {3 \choose 3} \zeta^3 + 
    {4 \choose 3} \beta \zeta^4 +
     {5 \choose 3} \beta^2 \zeta^5 + \dots \right) \\
     & \dots 
\end{align*}
\begin{align*}
    S_{k-1} & = \alpha^{k-1} \delta \sum_{j \ge k} {j \choose k} \beta ^{j-k} \zeta^j \\
    & = \alpha^{k-1} \delta \left( {k \choose k} \zeta^k + 
    {k+1 \choose k} \beta \zeta^{k+1} +
     {k+2 \choose k} \beta^2 \zeta^{k+2} + \dots \right)
\end{align*}
Now we want to sort the terms with respect to orders of $\zeta$. We see that a term with $\zeta^s$ shows up in $S_r$ if $r < s$.
\begin{align*}
     O(\zeta): & \qquad  \delta {1 \choose 1} \beta^0\\
    O(\zeta^2): & \qquad \delta \left(
     {2 \choose 1} \beta + \alpha {2 \choose 2} \beta^0   \right) \\
    O(\zeta^3): & \qquad 
    \delta \left( {3 \choose 1} \beta^2 + \alpha {3 \choose 2} \beta +   
    \alpha^2 {3 \choose 3} \beta^0 \right) \\
    & \dots
\end{align*}
In general, we find
\begin{align*}
   & O(\zeta^n): & \qquad \delta \sum_{m=1}^{\min(n, k)}
    {n \choose m} \beta^{n-m} \alpha^{m-1}
    & \le \delta \beta^{n-1} \alpha^{-1} \sum_{m=1}^{\min(n, k)}
    {n \choose m} \alpha^{m} \\
    & & & \le \delta \beta^{n-1} \alpha^{-1} \left( (1+\alpha)^n -1 \right) \\
     & & & \le \delta \beta^{n-1} \alpha^{-1} \alpha n (1+\alpha_{0n})^{n-1} \\
     &&&= n \delta \beta^{n-1} (1+\alpha_{0n})^{n-1}
\end{align*}
where $1+\alpha_{0n} \in [1, 1+\alpha]$. The value of $\alpha_{0n}$ can be different for different $n$.  The third inequality follows from the intermediate value theorem. 

\begin{align*}
    \frac{\delta \zeta}{1-\zeta}
    \sum_{r=0}^{k-1}(\frac{(\alpha \zeta)^r}{(1- \beta \zeta)^{r+1}}  \le \sum_{r=0}^{k-1} S_r & \le \sum_{n=1}^{\infty} n \delta \beta^{n-1} (1+\alpha_{0n})^{n-1} \zeta^n 
\end{align*}

Remark: In order to reorder the terms in the last inequality we need unconditional convergence of the sum over the $S_r$. Unconditional convergence is equivalent to abslolute convergence in $\mathbb{R}^n$. As a finite sum over absolutely convergent series, the sum over the $S_r$ is absolutely convergent too.

\end{proof}

\subsection{Results for multiple levels} 
\label{subsec:lemmata_multiple_levels}

In this section we use the subsequent notational conventions: 
\begin{enumerate}
    \item To avoid cumbersome notation the constants are not indexed with the level $l$. Instead we assume that $C_{i} = \max\limits_{l}\{C_{i,l} \}$, where $l$ counts the levels.
    \item The propagators on the levels $1, \dots, L-1$ have at least an accuracy order of $p_c$, i e. $p_c = \min\limits_{l=1, \dots, L-1}\{p_l\}$
\end{enumerate}

In the lemmata \ref{lemma:A5}, \ref{lemma:A6}, \ref{lemma:Lemma7} and \ref{lem:representations_for_E_1_-E_2_}, we assume that we have a coarsening factor $N$ which relates the levels and is described in \cref{subsec:multi-level_result}. Then the timesteps on the different levels obey the following relations:
\begin{equation}
\label{eq:relation_for_Delta_T_l_1}
    \Delta T_l = N^l \Delta T_0 ,
\end{equation}
where $\Delta T_0$ is the time step on the finest level, level 0. Moreover we have
\begin{equation}
\label{eq:relation_for_Delta_T_l_2}
    \Delta T_l = N^{l-(L-1)} \Delta T_{L-1} ,
\end{equation}
where $\Delta T_{L-1} $ is the time step on the coarsest level, level $L-1$.

\begin{lemma}
\label{lemma:first_estimate_theorem2}
We consider the recursively formulated inequality
\begin{equation*}
   e_{n,l}^{k_l} \le E_{l} + A_{l} \delta_{l-1} ,
\end{equation*}
with $\delta_{l-1} = e_{N,l-1}^{k_{l-1}}$ and 
$e_{n,l}^{k_{l}} \le e_{N,l}^{k_{l}}$ for all $l$, $n \le N$. Then we can show the following inequality
\begin{equation}
    e_{n,l}^{k_l} \le \sum_{\bar l =1}^l E_{ \bar l} \prod_{j= \bar l +1}^l A_{j} + \delta_0 \prod_{ \bar l = 1 }^l A_{ \bar l},
\end{equation}
which depends on $\delta_0$ and is independent of $\delta_{l-1}$ (if $l-1 \ge 1$).
\end{lemma}

\begin{proof}
The proof is given by induction. First we consider the initial case where $l=2$:
\begin{align*}
   e_{n,1}^{k_1} & \le E_{1} + A_{1} \delta_0 \\
   & = \sum_{\bar l =1}^1 E_{ \bar l} \prod_{j= \bar l +1}^1 A_{j} + \delta_0 \prod_{ \bar l = 1 }^1 A_{ \bar l},
\end{align*}

Now we do the induction step. We have
\begin{equation*}
    e_{n,l+1}^{k_{l+1}} \le E_{l+1} + A_{l+1} \delta_{l}.
\end{equation*}
Applying $\delta_{l} = e_{N,l}^{k_l}$ gives
\begin{align*}
e_{n,l+1}^{k_l+1} & \le E_{l+1} + A_{l+1} e_{N,l}^{k_l} \\
& \le E_{l+1} + A_{l+1} \left( \sum_{\bar l =1}^l E_{ \bar l} \prod_{j= \bar l +1}^l A_{j} + \delta_0 \prod_{ \bar l = 1 }^l A_{ \bar l}\right) \\
& =  E_{l+1} + \left(  \sum_{\bar l =1}^l E_{ \bar l} \prod_{j= \bar l +1}^{l+1} A_{j} + \delta_0 \prod_{ \bar l = 1 }^{l+1} A_{ \bar l}\right) \\
& = \sum_{\bar l =1}^{l+1} E_{ \bar l} \prod_{j= \bar l +1}^{l+1} A_{j} + \delta_0 \prod_{ \bar l = 1 }^{l+1} A_{ \bar l}
\end{align*}

\end{proof}

\begin{lemma}
\label{lemma:A5}
Suppose $\delta_0$ and $A_l$ satisfy the following relations
\begin{align*}
    \delta_0 & = c \frac{\Delta T_{L-1}^{p_c+1}}{N^{(L-1)(p_c+1)-1}}\\
    A_l & \le N \exp\left(C_2 \frac{ T}{N^{L-1-l}} + C_1 T \frac{\Delta T_{L-1}^{p_c} }{N^{(L-1-l)(p_c+1)}} \right), 
\end{align*}
(see \cref{lem:representations_for_E_1_-E_2_} for $A_l$).
Then, we can find the following bound
\begin{equation*}
    \delta_0 \prod_{ \bar l = 1 }^{L-1} A_{ \bar l} \le
    cT \Delta T_0^{p_0}
     \exp \left(C_2 T \frac{1-1/N^L}{1-1/N} + C_1 T \Delta T_{L-1}^{p_c} \frac{1-1/N^{L(p_c+1)}}{1-1/N^{p_c+1}}\right) .
\end{equation*}
\end{lemma}

\begin{proof}
We use the estimate for $A_{l}$ from  \cref{lem:representations_for_E_1_-E_2_}:
\begin{align*}
     \delta_0 \prod_{ \bar l = 1 }^{L-1} A_{ \bar l} & \le c \frac{\Delta T_{L-1}^{p_0+1}}{N^{(L-1)(p_0+1)-1}} \prod_{l=1}^{L-1} N \exp \left( \frac{C_2 T}{N^{L-1-l}} + \frac{C_1 T \Delta T_{L-1}^{p_c}}{N^{(p_c+1)(L-1-l)}} \right) \\
     \begin{split}
     & \le c \frac{\Delta T_{L-1}^{p_0+1}N^{L-1}}{N^{(L-1)(p_0+1)-1}}    \exp \Big(\ C_2 T \left( 1+1/N + \dots 1/N^{L-1} \right) + \\& \qquad  C_1 T \Delta T_{L-1}^{p_c} \left(1 + 1/N^{p_c+1} + \dots +  1/N^{(p_c+1)(L-1)} \right) \Big)\ 
       \end{split} \\
     & =c \frac{\Delta T_{L-1}^{p_0+1}N^{L-1}}{N^{(L-1)(p_0+1)-1}}
     \exp \left(C_2 T \frac{1-1/N^L}{1-1/N} + C_1 T \Delta T_{L-1}^{p_c} \frac{1-1/N^{L(p_c+1)}}{1-1/N^{p_c+1}}\right) \\
     & = cT \left(\frac{\Delta T_{L-1}}{N^{L-1}} \right)^{p_0}
     \exp \left(C_2 T \frac{1-1/N^L}{1-1/N} + C_1 T \Delta T_{L-1}^{p_c} \frac{1-1/N^{L(p_c+1)}}{1-1/N^{p_c+1}}\right). 
\end{align*}
In the third line the finite geometric series was applied. Furthermore, the relation $N \Delta T = T$ was used in the last line.

{\color{white} 1}
\end{proof}

\begin{lemma}
\label{lemma:A6}
Suppose $E_l$ and $A_l$ can be bounded by the relations \eqref{eq:bound_E_l} and \eqref{eq:bound_A_l}. Then the following estimate can be shown
\begin{align*}
     \sum_{l=1}^{L-1} E_{l} \prod_{j=l+1}^{L-1} A_{j} \le  &
    \exp \left(\frac{C_2 T}{1-1/N} +  \frac{C_1 T \Delta T_{L-1}}{1-1/N^{p_c+1}} \right)  \\ 
    &  
   \max_l \left({N \choose k_l +1} C_3 C_1^{k_l} \right) 
     \quad \frac{\Delta T_{L-1}^{k p_c +k+p_c} }{1-(1/N)^{k p_c +k+p_c}}. 
\end{align*}
\end{lemma}

\begin{proof}
We use the result from \cref{lemma:Lemma7}.
\begin{align*}
    \sum_{l=1}^{L-1} E_{l} \prod_{j=l+1}^{L-1} A_{j} & \le  
    \sum_{l=1}^{L-1} {N \choose k_l +1} C_3 C_1^{k_l} \Delta T_l^{k_l p_c +k_l+p_c} \Delta T_{L-1} \\
    & \quad \exp \left( C_2 T \frac{1-1/N^{L-1-l}}{1-1/N} + C_1 T \Delta T_{L-1}^{p_c} \frac{1-1/N^{(p_c+1)(L-1-l)}}{1-1/N^{p_c+1}} \right) \\
    & \le \sum_{l=1}^{L-1} {N \choose k_l +1} C_3 C_1^{k_l} \Delta T_l^{k_l p_c +k_l+p_c} \Delta T_{L-1} \quad \\
    & \quad \exp \left(  \frac{C_2 T}{1-1/N} + C_1  \frac{T \Delta T_{L-1}^{p_c}}{1-1/N^{p_c+1}} \right) \\
    & \le \exp \left(  \frac{C_2 T}{1-1/N} + C_1  \frac{T \Delta T_{L-1}^{p_c}}{1-1/N^{p_c+1}} \right)
    \Delta T_{L-1} \\
    & \quad \max_l \left({N \choose k_l +1} C_3 C_1^{k_l} \right) 
     \quad \sum_{l=1}^{L-1} \Delta T_l^{k_l p_c +k_l+p_c} \\
     & \le \quad \exp \left(  \frac{C_2 T}{1-1/N} + C_1  \frac{T \Delta T_{L-1}^{p_c}}{1-1/N^{p_c+1}} \right) \\
      & \quad \max_l \left({N \choose k_l +1} C_3 C_1^{k_l} \right) 
     \quad \frac{\Delta T_{L-1}^{k p_c +k+p_c} }{1-(1/N)^{k p_c +k+p_c}}
\end{align*}
where we used $k = \min\limits_{l=1, \dots, L-1} \{k_l\}$ and the geometric series in the last inequality.
\end{proof}

\begin{lemma}
\label{lemma:Lemma7}
Let $E_l$ and $A_l$ be bounded by the relations \eqref{eq:bound_E_l} and \eqref{eq:bound_A_l}.
For the expression $E_{l} \prod_{j=  l +1}^{L-1} A_{j}$ we find the following estimate
\begin{equation*}
\begin{split}
    E_{  l} \prod_{j=  l +1}^{L-1} A_{j} & \le
    {N \choose k_l+1} C_1^{k_l} C_3
    (\Delta T_l)^{(k_l p_c + k_l + p_c)} \\
    & \qquad \Delta T_{L-1}
    \exp\left(C_2 T \frac{1-1/N^{(L-1-l)}}{1-1/N} + C_1 T \Delta T_{L-1}^{p_c} \frac{1-1/N^{(p_c+1)(L-1-l)}}{1-1/N^{(p_c+1)}} \right) .
    \end{split}
\end{equation*}
\end{lemma}

\begin{proof}
The bounds \eqref{eq:bound_E_l} and \eqref{eq:bound_A_l} from lemma \ref{lem:representations_for_E_1_-E_2_} shall be used:
\begin{align*}
    E_{l} \prod_{j=  l +1}^{L-1} A_{j} & = 
    {N \choose k_l+1} C_1^{k_l} C_3 (\Delta T_l)^{(k_l+1)(p_c+1)} 
    \exp \left(C_2 \frac{T}{N^{L-1-l}} \right)
    \\ & \qquad
    N^{L-1-l} \exp\Big(\ C_2 T (1+1/N + \dots 1/N^{L-1-l-1}) \\ & \qquad \qquad \qquad  \qquad + C_1 T \Delta T_{L-1}^{p_c} (1+1/N^{p_c+1}+ \dots 
    1/N^{(p_c+1)(L-1-l-1)}) \Big)\ \\
    & = {N \choose k_l+1} C_1^{k_l} C_3
    (\Delta T_l)^{(k_l+1)(p_c+1)} N^{L-1-l} \\
    & \qquad \exp\left(C_2 T \frac{1-1/N^{(L-1-l)}}{1-1/N} + C_1 T \Delta T_{L-1}^{p_c} \frac{1-1/N^{(p_c+1)(L-1-l)}}{1-1/N^{(p_c+1)}} \right) \\
    & \le {N \choose k_l+1} C_1^{k_l} C_3
    (\Delta T_l)^{(k_l p_c + k_l + p_c)} \frac{\Delta T_{L-1}}{N^{L-1-l}} N^{L-1-l}  \\
    & \qquad \exp\left(C_2 T \frac{1-1/N^{(L-1-l)}}{1-1/N} + C_1 T \Delta T_{L-1}^{p_c} \frac{1-1/N^{(p_c+1)(L-1-l)}}{1-1/N^{(p_c+1)}} \right) \\
    & = {N \choose k_l+1} C_1^{k_l} C_3
    (\Delta T_l)^{(k_l p_c + k_l + p_c)} \Delta T_{L-1} \\
    & \qquad \exp\left(C_2 T \frac{1-1/N^{(L-1-l)}}{1-1/N} + C_1 T \Delta T_{L-1}^{p_c} \frac{1-1/N^{(p_c+1)(L-1-l)}}{1-1/N^{(p_c+1)}} \right) .
\end{align*}

\end{proof}

\begin{lemma}
\label{lem:representations_for_E_1_-E_2_}
Using the relations \eqref{eq:E_l__} and \eqref{eq:A_l__} to define $E_{l}$ and $A_{l}$, we can show the following bounds
\begin{align}
    E_{l} &\le {N \choose k_l+1} C_1^{k_l} C_3 (\Delta T_l)^{(k+1)(p_c+1)} 
    \exp \left(C_2 \frac{T}{N^{L-1-l}} \right)
    \label{eq:bound_E_l}
\\
    A_{l} &\le N \exp\left(C_2 \frac{ T}{N^{L-1-l}} + C_1 T \frac{\Delta T_{L-1}^{p_c} }{N^{(L-1-l)(p_c+1)}} \right) , \label{eq:bound_A_l}
\end{align}
where the propagator on level $l$ for $l \ge 1$ has an accuracy order of at least $p_c$.
\end{lemma}

\begin{proof}
According to the relations \eqref{eq:E_l__} and \eqref{eq:A_l__} we have
\begin{align*}
     E_l &= {n \choose k_l+1}    \gamma_l \alpha^{k_l}_l  \beta^{N-{k_l}-1}_l \label{eq:E_l_} \\
    A_l &=  N \beta^{N-1}_l  (1+\alpha_{0N,l})^{N-1}.
\end{align*}

In addition, we assume we have a coarsening factor $N$ which relates the levels and is described in \cref{subsec:multi-level_result} and by the relations \eqref{eq:relation_for_Delta_T_l_1} and eq\ref{eq:relation_for_Delta_T_l_2}.
Then the following identities hold
\begin{align*}
    \alpha_l = C_1 \Delta T_l^{p_c+1} & = C_1 \left( \frac{\Delta T_{L-1}}{N^{L-1-l}} \right)^{p_c+1} \\
    \beta_l = 1+C_2 \Delta T_l & = \left( 1+C_2 \frac{\Delta T_{L-1}}{N^{L-1-l}} \right)\\
    \gamma_l = C_3 \Delta T_l^{p_c+1} & = C_3 \left( \frac{\Delta T_{L-1}}{N^{L-1-l}} \right)^{p_c+1}\\
    \delta_0 = c \Delta T_1 \Delta t^{p_0} & = c  
    \left( \frac{\Delta T_{L-1}}{N^{L-1-1}} \right)
    \left(\frac{\Delta T_{L-1}}{N^{L-1}} \right)^{p_0}
    = c \frac{\Delta T_{L-1}^{p_0+1}}{N^{(L-1)(p_0+1)-1}}
\end{align*}

Thus, we get for $E_{l}$
\begin{align*}
    E_{l} & = {N \choose k_l+1} 
    C_3 \left( \frac{\Delta T_{L-1}}{N^{L-1-l}} \right)^{p_c+1}
    \left(C_1 \left( \frac{\Delta T_{L-1}}{N^{L-1-l}} \right)^{p_c+1} \right)^{k_l}
    \left(1+C_2 \frac{\Delta T_{L-1}}{N^{L-1-l}} \right)^{N-k_l-1}
    \\
    & \le {N \choose k_l+1} C_1^{k_l} C_3 (\Delta T_l)^{(k+1)(p_c+1)}
    \exp \left(C_2 \frac{T_n - T_{k_l-1}}{N^{L-1-l}} \right) 
    \\
    & \le {N \choose k_l+1} C_1^{k_l} C_3 (\Delta T_l)^{(k+1)(p_c+1)}
    \exp \left(C_2 \frac{T}{N^{L-1-l}} \right) ,
\end{align*}
where $T_0, T_1, \dots, T_N$ denote the grid points on the coarsest level.

In addition, for $A_{l}$ we find
\begin{align*}
   A_{l} & =  N \beta_l^{N-1} (1+ \alpha_{0N,l})^{N-1} \\
   & \le N \beta_l^{N-1} (1+ \alpha_{l})^{N-1} \\
   & = N \left(1+C_2 \frac{\Delta T_{L-1}}{N^{L-1-l}} \right)^{N-1} 
   \left(1+ C_1 \left( \frac{\Delta T_{L-1}}{N^{L-1-l}} \right)^{p_c+1} \right)^{N-1} \\
   & \le N \exp\left(C_2 (N-1)\frac{\Delta T_{L-1}}{N^{L-1-l}} + C_1 (N-1) \left(\frac{\Delta T_{L-1}}{N^{L-1-l}}\right)^{p_c+1} \right) \\
   & \le N \exp\left(C_2 \frac{ T}{N^{L-1-l}} + C_1 T \left(\frac{\Delta T_{L-1}^{p_c} }{N^{(L-1-l)(p_c+1)}}\right) \right)
\end{align*}
In the last inequality we exploit that $N \Delta T_{L-1} = T$, where $T$ is the length of the integration interval (on the coarsest level).
\end{proof}

\section{Details about $\mathcal{M}_{0,l}$ and $\mathcal{M}_{1,l}$} 
\label{app:sec:convergence_proof_with_averaging}

To simplify the notation, we write $\mathcal{M}_{0}$ instead of $\mathcal{M}_{0,l}$ and $\mathcal{M}_{1}$ instead of $\mathcal{M}_{1,l}$. Additionally, we neglect the level when we refer to the exact and numerical propagators of the averaged system, i.e. we use $\bar E$ instead of $\bar E^l$ and $\bar G$ instead of $\bar G^l$. To emphasize the dependence of a solution of a differential equation on the initial values we also employ the flow notation in this section, consequently  $\varphi_t(u_0)$ and $u(t)$ are used equivalently. 
The subsequent investigations shall justify the assumption of Lipschitz continuity in $u_0$ of

\begin{equation} \label{eq:definitions_of_M0_and_M1}
\begin{split}
\mathcal{M}_0 (u_0, \epsilon, \eta, \Delta T) &= \bar E(u_0) - \bar G(u_0)\qquad \text{ and }  \\
    \mathcal{M}_1 (u_0, \epsilon, \eta) &= E(u_0) - \bar E(u_0). 
\end{split}
\end{equation}

\begin{lemma}
\label{lemma:Lemma_about_M0}
$\mathcal{M}_0$ is Lipschitz continuous in the first argument. (Parts of the proof are needed later in the next Lemma.)
\end{lemma}

\begin{proof}

For the estimate of $\mathcal{M}_0$, we apply equation \eqref{eq:(7)_Gander_Hairer} and 
follow the same arguments as in the proof of  \cref{thm:two-level_bound}. This gives  
\begin{equation*}
\| \mathcal{M}_0 (v_1(t_0), \epsilon, \eta, \Delta T) - \mathcal{M}_0 (v_2(t_0), \epsilon, \eta, \Delta T) \| \le
 \Delta T^{p+1} L \| v_1(t) -v_2(t) \|.
\end{equation*}

Now, we have to show that the expression $\| v_1(t) -v_2(t) \|$ satisfies an estimate of the form
\begin{equation*}
    \| v_1(t) -v_2(t) \| \le \tilde L \| v_1(t_0) -v_2(t_0) \|,
\end{equation*}
where $t_0$ is the initial time.

First, we show that the averaged non-linearity can be bounded:

Let us assume that the unaveraged non-linearity and its derivatives are bounded by a constant $M$. In particular we have
\begin{equation*}
   -M \le \partial_2 N(t, \varphi_t(x)) \le M ,
\end{equation*}
where $\partial_2$ denotes the partial derivative with respect to the second component.
From this we can conclude that
\begin{align*}
   -M & = - \frac{1}{\eta} \int_{-\eta/2}^{\eta/2} \rho \left(\frac{s}{\eta}\right) M \ ds \\
   & \le \frac{1}{\eta} \int_{-\eta/2}^{\eta/2} \rho \left(\frac{s}{\eta}\right) \partial_2 N(s+t, \varphi_{s+t}(x)) \ ds \\
   & = \partial_2 N_{\eta}(t, \varphi_{t}(x)) \\
   & \le \frac{1}{\eta} \int_{-\eta/2}^{\eta/2} \rho \left(\frac{s}{\eta}\right) M \ ds = M .
\end{align*}

In \cite{Wilke_Pruss_2010} on page 92 equation (4.22) we find a result about the dependence of an ordinary differential equation on the initial condition. Especially, the dependence of a differential equation on the initial condition obeys another differential equation
\begin{equation}
\label{eq:dep_on_initial_data}
    \partial_t (\partial_x \varphi_t(x)) = \partial_2 N_{\eta}(t, \varphi_t(x)) \partial_x  \varphi_t(x), 
\end{equation}
where the identity matrix $I$ is the initial condition of equation \eqref{eq:dep_on_initial_data}.

This is a linear equation in $\partial_x \varphi_t(x)$. 
Thus, applying the bound from above we can conclude 
\begin{equation*}
    - e^{Mt} \le \partial_x \varphi_t(x) \le e^{Mt} ,
\end{equation*}
which leads to the relation 
\begin{equation*}
    |\varphi_t(v_0) - \varphi_t(w_0)| \le |v_0 -w_0| e^{Mt} .
\end{equation*}
Only the non-stiff nonlinearity determines how fast two solutions of the modulation equation with distnict initial values diverge! This is also true for the averaged version of the modulation equation!

{\color{white}1}
\end{proof}


We rewrite equation \eqref{eq:modulation_equation}
and \eqref{eq:averaged_problem} as
\begin{align}
\frac{dw}{d\tilde t} &= N \left( \frac{\tilde t}{\epsilon},w(\tilde t) \right) & \tilde t \in [0,1] \\
\frac{d \bar w}{d\tilde t} &= N_{\eta} \left( \frac{\tilde t}{\epsilon},\bar w(\tilde t) \right) & \tilde t \in [0,1].
\end{align}
Applying the time transformation $t (\tilde t) =   \tilde t /\epsilon$ leads to the new problems
\begin{align}
\frac{dw_{s}}{d t} &= \epsilon N \left( t,w_{s}(t) \right) & t \in [0,1/\epsilon] \label{eq_tranformed1}\\
\frac{d \bar w_{s}}{d t} &= \epsilon N_{\eta} \left(  t,\bar w_{s}( t) \right) & t \in [0,1/\epsilon]. \label{eq_tranformed2}
\end{align}
From now on, we will work with equations \cref{eq_tranformed1} and \cref{eq_tranformed2} and skip the index $s$ (for slow).
Considering the equations \cref{eq_tranformed1} and \cref{eq_tranformed2} we see that the derivatives of $w$ and $\bar w$ are bounded by $\epsilon M$. In the following, we will use the equivalent notations $w(t) = \varphi_t(w_0)$ and $\bar w(t) = \bar  \varphi_t(w_0)$ .  Moreover, the subsequent expression is valid
\begin{align*}
    \mathcal{M}_1(v_0, \epsilon, \eta) - \mathcal{M}_1(w_0, \epsilon, \eta) 
    & = v(t) - \bar v(t) -[w(t) - \bar w(t)] \\
    & = \epsilon \int_0^t N(\tau, \varphi_{\tau}(v_0)) d \tau - \epsilon \int_0^t N_{\eta} (\tau, \bar \varphi_{\tau}(v_0)) d \tau \\
    & \quad - \left( \epsilon \int_0^t N(\tau, \varphi_{\tau}(w_0)) d \tau - \epsilon \int_0^t N_{\eta} (\tau, \bar \varphi_{\tau}(w_0)) d \tau \right) .
\end{align*}
We will now establish the Lipschitz continuity of $\mathcal{M}_1(u_0, \epsilon, \eta)$ in the first component.

\begin{lemma}
\label{lemma:Lemma_about_M1}
$\mathcal{M}_1$ can be bounded by a Lipschitz constant of the order $\epsilon \eta$ in $u_0$.
\end{lemma}

\begin{proof}

For the estimate of $\mathcal{M}_1$ we have
\begin{align*}
    \varphi_t(v_0) - \bar \varphi_t(v_0)
    &= \epsilon \int_0^t N(\tau, \varphi_{\tau}(v_0)) d \tau - \epsilon \int_0^t N_{\eta} (\tau, \bar \varphi_{\tau}(v_0)) d \tau \\
    &= \epsilon \int_0^t \frac{1}{\eta} \int_{-\eta/2}^{\eta/2} \rho \left( \frac{s}{\eta} \right) \left( N(\tau, \varphi_{\tau}(v_0)) - N(\tau+s, \bar \varphi_{\tau}(v_0)) \right) \ d s \ d \tau \\
    &= \dots
\end{align*}

A zeroth order Taylor representation in $s$ is applied for $N(\tau+s, v(\tau))$ :
\begin{align*}
    N(\tau+s, \bar \varphi_{\tau}(v_0)) 
    &= N(\tau, \bar \varphi_{\tau}(v_0)) +
    \int_{0}^s \partial_1 N(\tau+r, \bar \varphi_{\tau}(v_0)) d r .
\end{align*}
Note: $s \in [-\eta/2, \eta/2]$.

\begin{align*}
    \dots &= \epsilon \int_0^t \frac{1}{\eta} \int_{-\eta/2}^{\eta/2} \rho \left( \frac{s}{\eta} \right) \left( N(\tau, \varphi_{\tau}(v_0)) - \left[N(\tau, \bar \varphi_{\tau}(v_0)) +
    \int_{0}^s \partial_1 N(\tau+r, \bar \varphi_{\tau}(v_0)) \ d r \right] \right) \ d s \ d \tau \\
    &=\epsilon \int_0^t \frac{1}{\eta} \int_{-\eta/2}^{\eta/2} \rho \left( \frac{s}{\eta} \right) \bigg(  N(\tau, \varphi_{\tau}(v_0)) - N(\tau, \bar \varphi_{\tau}(v_0)) \bigg) \ d s \ d \tau \dots \text{ (term 1)} \\
    & \qquad - \epsilon \int_0^t \frac{1}{\eta} \int_{-\eta/2}^{\eta/2} \rho \left( \frac{s}{\eta} \right) 
    \int_{0}^s \partial_1 N(\tau+r,\bar \varphi_{\tau}(v_0)) \ d r  \ d s \ d \tau \text{ (term 2)}\\
\end{align*}

We split the above relation into term 1 and term 2 and continue the computations separately. 
For the contribution of term 2 we find the estimate 

\begin{align*}
    &\|\varphi_t(v_0) - \bar \varphi_t(v_0) - (\varphi_t(w_0) - \bar \varphi_t(w_0))\| \Big\vert_{\text{term 2}}  \\ 
    & \qquad \le \epsilon \int_0^t \frac{1}{\eta} \int_{-\eta/2}^{\eta/2} \rho \left( \frac{s}{\eta} \right) 
    \int_{0}^s \| \partial_1 N(\tau+r,\bar \varphi_{\tau}(v_0)) -  \partial_1 N(\tau+r,\bar \varphi_{\tau}(w_0)) \| \ d r  \ d s \ d \tau \\
    & \qquad \le \epsilon \int_0^t \frac{1}{\eta} \int_{-\eta/2}^{\eta/2} \rho \left( \frac{s}{\eta} \right) 
    \int_{0}^s M \| \bar \varphi_{\tau}(v_0) -\bar \varphi_{\tau}(w_0) \| \ d r  \ d s \ d \tau \\
    & \qquad \le \epsilon \int_0^t \frac{1}{\eta} \int_{-\eta/2}^{\eta/2} \rho \left( \frac{s}{\eta} \right) 
    M s e^{\epsilon M \tau} \| v_0 - w_0 \| \ d s \ d \tau \\
    & \qquad \le  \epsilon \int_0^t M e^{\epsilon M \tau}\eta \ d \tau \| v_0 - w_0 \|\\
    & \qquad \le \eta \frac{M \epsilon}{M \epsilon}  (e^{\epsilon M t}-1)   \| v_0 - w_0 \| \\
  & \qquad \le M \eta \epsilon e^{\epsilon M t}  \| v_0 - w_0 \| 
\end{align*}
In the last line, the mean value theorem is applied.
This bound is of order $\epsilon \eta$ and linear in $ \| v_0 - w_0 \|$. Note: $t \in [0, 1/\epsilon]$

For term 1 we find
\begin{equation}
\label{eq:M1_in_estimate}
    \begin{split}
    &\|\varphi_t(v_0) - \bar \varphi_t(v_0) - (\varphi_t(w_0) - \bar \varphi_t(w_0))\| \Big\vert_{\text{term 1}}  \\ 
    & \qquad =
    \epsilon \left\|\int_0^t \underbrace{ \frac{1}{\eta} \int_0^{\eta} \rho \left( \frac{s}{\eta} \right)}_{=1} \bigg( N(\tau, \varphi_{\tau}(v_0)) - N(\tau, \bar \varphi_{\tau}(v_0)) - \left[N(\tau, \varphi_{\tau}(w_0)) - N(\tau, \bar \varphi_{\tau}(w_0))\right] \bigg) \ d s \ d \tau \right\| \\
     & \qquad =
    \epsilon \left\|\int_0^t   N(\tau, \varphi_{\tau}(v_0)) - N(\tau, \bar \varphi_{\tau}(v_0)) - \left[N(\tau, \varphi_{\tau}(w_0)) - N(\tau, \bar \varphi_{\tau}(w_0))\right]  \ d \tau \right\| 
        \end{split}
\end{equation}


To estimate the contribution of term 1, we consider 
\begin{align*}
    N(\tau, \varphi_{\tau}(v_0)) - N(\tau, \bar \varphi_{\tau}(v_0)) .
\end{align*}

Again, a Taylor expansion for this expression can be applied
\begin{align*}
    N(\tau, \varphi_{\tau}(v_0)) = & N(\tau, \bar \varphi_{\tau}(v_0)) +  (\varphi_{\tau}(v_0) - \bar \varphi_{\tau}(v_0)) \partial_2 N(\tau, \bar \varphi_{\tau}(v_0)) + \\ & \quad\int_{\bar \varphi_{\tau}(v_0)}^{ \varphi_{\tau}(v_0)}
    ( \varphi_{\tau}(v_0)-\sigma) \partial_2^2 N(\tau, \sigma) d \sigma ,
\end{align*}
where $\partial_2$ denotes the partial derivative with respect to the second component.
Thus,
\begin{align*}
    &N(\tau, \varphi_{\tau}(v_0)) - N(\tau, \bar \varphi_{\tau}(v_0)) = \\
    & \qquad (\varphi_{\tau}(v_0) - \bar \varphi_{\tau}(v_0)) \partial_2 N(\tau, \bar \varphi_{\tau}(v_0)) + \int_{\bar \varphi_{\tau}(v_0)}^{ \varphi_{\tau}(v_0)}
    ( \varphi_{\tau}(v_0)-\sigma) \partial_2^2 N(\tau, \sigma) d \sigma
\end{align*}

and
\begin{align*}
    &N(\tau, \varphi_{\tau}(v_0)) - N(\tau, \bar \varphi_{\tau}(v_0)) - \left[N(\tau, \varphi_{\tau}(w_0)) - N(\tau, \bar \varphi_{\tau}(w_0)) \right] = \\
    & \qquad (\varphi_{\tau}(v_0) - \bar \varphi_{\tau}(v_0)) \partial_2 N(\tau, \bar \varphi_{\tau}(v_0)) + \int_{\bar \varphi_{\tau}(v_0)}^{ \varphi_{\tau}(v_0)}
    ( \varphi_{\tau}(v_0)-\sigma) \partial_2^2 N(\tau, \sigma) d \sigma - \dots \\
    &\qquad \left[ (\varphi_{\tau}(w_0) - \bar \varphi_{\tau}(w_0)) \partial_2 N(\tau, \bar \varphi_{\tau}(w_0)) + \int_{\bar \varphi_{\tau}(w_0)}^{\varphi_{\tau}(w_0)}
    ( \varphi_{\tau}(w_0)-\sigma) \partial_2^2 N(\tau, \sigma) d \sigma \right] =\\
    &\qquad \iota_1 (\varphi_{\tau}(v_0) - \bar \varphi_{\tau}(v_0)) - \left[\iota_2(\varphi_{\tau}(w_0) - \bar \varphi_{\tau}(w_0)) \right]
    + \dots \\ & \qquad \int_{\bar \varphi_{\tau}(v_0)}^{ \varphi_{\tau}(v_0)}
    ( \varphi_{\tau}(v_0)-\sigma) \partial_2^2 N(\tau, \sigma) d \sigma -\int_{\bar \varphi_{\tau}(w_0)}^{\varphi_{\tau}(w_0)}
    ( \varphi_{\tau}(w_0)-\sigma) \partial_2^2 N(\tau, \sigma) d \sigma ,
\end{align*}
where we introduced the notation
\begin{equation*}
    \iota_1(a) = a \partial_2 N(\tau, \bar \varphi_{\tau}(v_0)) \qquad \qquad \iota_2(a) = a \partial_2 N(\tau, \bar \varphi_{\tau}(w_0))
\end{equation*}

An estimate for the integrals in the above expression is given by
\begin{align*}
 &\int_{\bar \varphi_{\tau}(v_0)}^{\varphi_{\tau}(v_0)}
    \underbrace{( \varphi_{\tau}(v_0)-\sigma)}_{O(\epsilon \eta)} \partial_2^2 N(\tau, \sigma) d \sigma 
    -  \int_{\bar \varphi_{\tau}(w_0)}^{\varphi_{\tau}(w_0)}
    \underbrace{( \varphi_{\tau}(w_0)-\sigma)}_{O(\epsilon \eta)} \partial_2^2 N(\tau, \sigma) d
    \sigma \\
    &\qquad \le M \epsilon \eta \left( \int_{\bar  \varphi_{\tau}(v_0)}^{\varphi_{\tau}(v_0)} d s - \int_{\bar \varphi_{\tau}(w_0)}^{\varphi_{\tau}(w_0)} d s \right) \\
    &\qquad \le M \epsilon \eta \left[(\varphi_{\tau}(v_0) - \varphi_{\tau}(w_0)) - (\bar \varphi_{\tau}(v_0) - \bar \varphi_{\tau}(w_0)) \right] .
\end{align*}

For the other terms we have
\begin{align*}
    & \iota_1 [\varphi_{\tau}(v_0)- \bar \varphi_{\tau}(v_0)] - \iota_2 [\varphi_{\tau}(w_0) - \bar \varphi_{\tau}(w_0)]   \\
    & \qquad = \iota_1 [\varphi_{\tau}(v_0) - \bar \varphi_{\tau}(v_0)] - \iota_2 [\varphi_{\tau}(v_0) - \bar \varphi_{\tau}(v_0)] + \iota_2 [\varphi_{\tau}(v_0) - \bar \varphi_{\tau}(v_0)] - \iota_2 [\varphi_{\tau}(w_0) - \bar \varphi_{\tau}(w_0)] \\
    & \qquad = \underbrace{(\iota_1 - \iota_2)}_{ \le M e^{\epsilon M \tau} |v_0 -w_0|} \underbrace{[\varphi_{\tau}(v_0) - \bar \varphi_{\tau}(v_0)]}_{\le \tilde m \epsilon \eta = O(\epsilon \eta)} + 
    \underbrace{\iota_2 \left\{ [\varphi_{\tau}(v_0) - \bar \varphi_{\tau}(v_0)] - [\varphi_{\tau}(w_0) - \bar \varphi_{\tau}(w_0)] \right\}}_{\le M ([\varphi_{\tau}(v_0) - \bar \varphi_{\tau}(v_0)] - [\varphi_{\tau}(w_0) - \bar \varphi_{\tau}(w_0)])} .
\end{align*}
For the $O(\eta \epsilon)$ estimate, we refer to Corollary 4.1 in \cite{Peddle_Haut_Wingate_2019} or equation \cref{eq_M_1l_est}.

In this place, the derived estimates can be used to apply the integral version of Gronwall's lemma together with boundedness of the derivatives of the nonlinearity $N$. In particular, Gronwall's inequality can be applied to

\begin{align*}
    & \varphi_t(v_0) - \bar \varphi_t(v_0) - [\varphi_t(w_0) - \bar \varphi_t(w_0)] \\
    & \qquad \le  M \epsilon \eta  e^{\epsilon M t} \| v_0 - w_0 \| \\
    &\quad \qquad + \epsilon \int_0^t M \epsilon \eta \left[(\varphi_{\tau}(v_0) - \varphi_{\tau}(w_0)) - (\bar \varphi_{\tau}(v_0) - \bar \varphi_{\tau}(w_0)) \right]  \ d \tau \\ 
    &\quad \qquad + \epsilon \int_0^t  M ([\varphi_{\tau}(v_0) - \bar \varphi_{\tau}(v_0)] - [\varphi_{\tau}(w_0) - \bar \varphi_{\tau}(w_0)])  \ d \tau \\
    &\quad \qquad + \epsilon \int_0^t \tilde  M e^{\epsilon M \tau} \epsilon \eta |v_0 -w_0|  \ d \tau \\
    & \qquad \le \underbrace{\left(M + \epsilon \tilde M \right) e^{\epsilon M t} \epsilon \eta  \| v_0 - w_0 \|}_{\alpha_G} \\
    & \quad\qquad + \epsilon \int_0^t M (1+ \epsilon \eta) \left[(\varphi_{\tau}(v_0) - \varphi_{\tau}(w_0)) - (\bar \varphi_{\tau}(v_0) - \bar \varphi_{\tau}(w_0)) \right]  \ d \tau ,
\end{align*}
where $t \in [0, 1/\epsilon]$. Again, the mean value theorem is used for the last inequality. As $\alpha_G$ depends linearly on $ \| v_0 - w_0 \|$, we obtain Lipschitz continuity in the initial data. Additionally $\alpha_G$ contains the factor $\epsilon \eta$ which guarantees that the bound of $\mathcal{M}_1$ has the same factor.

{\color{white} 1}
\end{proof}

\section*{Acknowledgments}
The authors would like to thank Rupert Klein for reading and discussing the manuscript.
Juliane Rosemeier is funded by German Research Foundation (DFG) through Walter Benjamin Programme, project {\it Formulation and numerical computation of the low frequency mean flow of fluids}.

\end{document}